\documentclass[11pt,amssymb]{amsart}
\usepackage{amssymb}
\usepackage[mathscr]{eucal}

\newcommand{\Z}{{\mathbb Z}}

\newcommand{\Q}{{\mathbb Q}}

\newcommand{\fR}{{\mathfrak R}}

\newcommand{\I}{{\mathcal O}}

\newcommand{\G}{{\mathcal G}}
\newcommand{\g}{{\mathfrak{g}}}
\newcommand{\fa}{{\mathfrak{a}}}

\newcommand{\B}{{\mathscr{B}}}

\newcommand{\Hom}{{\rm Hom}}

\newcommand{\ba}{\mbox{\boldmath{$\alpha$}}}
\newcommand{\bm}{\mbox{\boldmath{$\mu$}}}
\newtheorem{thm}{Theorem}[section]
\newtheorem{lemma}[thm]{Lemma}
\newtheorem{prop}[thm]{Proposition}
\newtheorem{cor}[thm]{Corollary}

\usepackage[all,cmtip]{xy}
\begin{document}

\title[Representations of algebraic groups and deformations]{On abstract representations of the groups of rational points of algebraic groups and their deformations}

\begin{abstract}
In this paper, we continue our study, begun in \cite{IR}, of abstract representations of elementary subgroups of Chevalley groups of rank $\geq 2.$ First, we extend the methods of \cite{IR} to analyze representations of elementary groups over arbitrary associative rings, and as a consequence, prove the conjecture of Borel and Tits \cite{BT} on abstract homomorphisms of the groups of rational points of algebraic groups for groups of the form ${\bf SL}_{n,D}$, where $D$ is a finite-dimensional central division algebra over a field of characteristic zero. Second, we apply the results of \cite{IR} to study deformations of representations of elementary subgroups of universal Chevalley groups of rank $\geq 2$ over finitely generated commutative rings.

\end{abstract}

\author[I.A.~Rapinchuk]{Igor A. Rapinchuk}

\address{Department of Mathematics, Yale University, New Haven, CT 06502}

\email{igor.rapinchuk@yale.edu}

\maketitle

\section{Introduction and statement of the main results}\label{S:I}

The goal of this paper is two-fold. First, we extend the
methods and results developed in our paper \cite{IR} to analyze abstract representations of Chevalley groups over commutative rings to elementary groups over arbitrary associative rings. As a consequence of this analysis, we prove the conjecture of Borel and Tits (\cite{BT}, 8.19) on abstract homomorphisms of the groups of rational points of algebraic groups for groups of the form $\mathbf{SL}_{n, D}$, where $D$ is a finite-dimensional central division algebra over a field of characteristic zero. Second, we apply the results of \cite{IR} to study deformations of representations of the elementary subgroup $\Gamma = E(\Phi, R)$ of a universal Chevalley group associated to a root system $\Phi$ of rank $\geq 2$ over a finitely generated commutative ring $R$. This relies on the description, obtained in \cite{IR}, of representations with {\it nonreductive image}, which are at the heart of the Borel-Tits conjecture (recall that representations with reductive image were completely described in \cite{BT}). We also use techniques of representation and character varieties (cf.~Lubotzky-Magid \cite{LM}), in conjunction with the fact that such $\Gamma$ satisfies Kazhdan's property (T), which was recently established in \cite{EJK}.

Before formulating of our first result, let us recall the statement of the Borel-Tits conjecture. As usual, for an algebraic $G$ defined over a field $k$, we will denote by $G^+$ the subgroup of $G(k)$ generated by the $k$-rational points of the unipotent radicals of the parabolic $k$-subgroups of $G$.
\vskip3mm

\noindent (BT) \ \ \parbox{15cm}{Let $G$ and $G'$ be algebraic groups defined over infinite fields $k$ and $k'$, respectively. If $\rho \colon G(k) \to G'(k')$ is any abstract homomorphism such that $\rho(G^+)$ is Zariski-dense in $G'(k'),$ then {\it there exists a commutative finite-dimensional $k'$-algebra $C$ and a ring homomorphism $f_C \colon k \to C$ such that  $\rho = \sigma \circ r_{C/k'} \circ F$, where $F \colon G(k) \to _{C}\!G(C)$ is induced by $f_C$ ($_{C}\!G$ is the group obtained by change of scalars), $r_{C/k'} \colon _{C}\!G(C) \to R_{C/k'} (_{C}\!G)(k')$ is the canonical isomorphism  (here $R_{C/k'}$ denotes the functor of restriction of scalars), and $\sigma$ is a rational $k'$-morphism of $R_{C/k'} (_{C}\!G)$ to $G'.$}}

\vskip3mm

If an abstract homomorphism $\rho \colon G(k) \to G'(k')$ admits a factorization as in (BT), we will say that $\rho$
has a {\it standard description}\footnotemark \footnotetext{It was pointed out to us by G.~Prasad that, due to the existence of exotic pseudo-reductive groups, constructed in \cite{CGP}, 7.2, one should probably require that char $k, k' \neq 2,3$ in the statement of (BT).}. Our result concerning (BT) is as follows. Given a finite-dimensional central division algebra $D$ over a field $k$, we let $G = \mathbf{SL}_{n,D}$ denote the algebraic $k$-group such that $G(k) = SL_{n} (D),$ the group of elements of $GL_n (D)$ having reduced norm one; recall that $G$ is an inner form of type $A_{\ell}$ (cf. \cite{KMRT} or \cite{PR} for the details).

\vskip2mm

\noindent {\bf Theorem 1.} {\it Let $D$ be a finite-dimensional central division algebra over a field $k$ of characteristic 0, and let $G = \mathbf{SL}_{n,D}$, where $n \geq 3.$ Let $\rho \colon G(k) \to GL_m (K)$ be a finite-dimensional linear representation of $G(k)$ over an algebraically closed field $K$ of characteristic 0, and set $H = \overline{\rho (G(k))}$ (Zariski-closure). Then the abstract homomorphism $\rho \colon G(k) \to H(K)$ has a standard description.}

\vskip2mm

In fact, we will see in \S \ref{S:NC} that a similar, but somewhat weaker, statement
can be established for representations of elementary groups over arbitrary associative rings, not just division algebras (see Theorem \ref{NC:T-NonComm2} for a precise statement). It should be observed that while the overall structure of the proof of Theorem 1 resembles that of the Main Theorem of \cite{IR}, the analogs of the $K$-theoretic results of Stein
\cite{St2}, which played a crucial role in \cite{IR}, were not available in the noncommutative setting. So, part of our argument is dedicated to developing the required $K$-theoretic results, which is done in \S 2 using the computations of relative $K_2$ groups given by Bak and Rehmann \cite{BR}.

As we have already mentioned, results describing representations of a given group $\Gamma$ with nonreductive image can be used to analyze {\it deformations} of representations of $\Gamma$, which is the second major theme of this paper. Formally, over a field of characteristic 0, deformations of (completely reducible) $n$-dimensional representations of a finitely generated group $\Gamma$ can be understood in terms of the corresponding character variety $X_n (\Gamma).$
For $\Gamma = E(\Phi, R)$, the elementary subgroup of $G(R)$, where $G$ is a universal Chevalley-Demazure group scheme corresponding to a reduced irreducible root system of rank $>1$ and $R$ is a finitely generated commutative ring, we use the results of \cite{IR} to estimate the dimension of $X_n (\Gamma)$ as a function of $n$ (we note that it was recently shown in \cite{EJK} that such $\Gamma$ possesses Kazhdan's property (T), hence is finitely generated, so the representation variety $R_n (\Gamma)$ and the associated character variety $X_n (\Gamma)$ are defined -- see \S \ref{S:D} for a brief review of these notions and \cite{LM} for complete details). To put our result into perspective, we recall that for $\Gamma = F_d,$ the free group on $d > 1$ generators, the dimension $\varkappa_n (\Gamma) := \dim X_n (\Gamma)$ is given by
$$
\varkappa_n (\Gamma) = (d-1)n^2 + 1,
$$
i.e. the growth of $\varkappa_n (\Gamma)$ is {\it quadratic} in $n.$ It follows that the rate of growth cannot be more than quadratic for {\it any} finitely generated group (and it is indeed quadratic in some important situations, such as $\Gamma = \pi_g$, the fundamental group of a compact orientable surface of genus $g > 1$, cf. \cite{RBC}). At the other end of the spectrum are the groups $\Gamma$, called $SS$-{\it rigid}, for which $\varkappa_n (\Gamma) = 0$ for all $n \geq 1$. For example, according to Margulis's Superrigidity Theorem (\cite{Ma}, Ch. VII, Theorem 5.6 and 5.25, Theorem A), all irreducible higher-rank lattices are $SS$-rigid (see \S \ref{S:SR} regarding the superrigidity of groups like $E(\Phi, \I),$ where $\I$ is a ring of algebraic integers). Now, the argument in (\cite{AR1}, \S 2) shows that if $\Gamma$ is not $SS$-rigid, then the rate of growth of $\varkappa_{\Gamma}(n)$ is at least linear. It follows that unless $\Gamma$ is $SS$-rigid, the growth rate of $\varkappa_n (\Gamma)$ is between linear and quadratic. Our result shows that for $\Gamma = E(\Phi, R)$ as above, this rate is the minimal possible, namely linear.

To formulate our result, we recall that a pair $(\Phi, R)$ consisting of a reduced irreducible root system of rank $> 1$ and a commutative ring $R$ was called {\it nice} in \cite{IR} if $2 \in R^{\times}$ whenever $\Phi$ contains a subsystem of type $B_2$, and $2, 3 \in R^{\times}$ if $\Phi$ is of type $G_2.$

\vskip2mm

\noindent {\bf Theorem 2.} {\it Let $\Phi$ be a reduced irreducible root system of rank $\geq 2$, $R$ a finitely generated commutative ring such that $(\Phi, R)$ is a nice pair, and $G$ the universal Chevalley-Demazure group scheme of type $\Phi$. Denote by $\Gamma = E(\Phi, R)$ the elementary subgroup of $G(R)$ and consider the variety $X_n (\Gamma)$ of characters of $n$-dimensional representations of $\Gamma$ over an algebraically closed field $K$ of characteristic 0. Then there exists a constant $c = c(R)$ (depending only on $R$) such that $\varkappa_n (\Gamma) := \dim X_n (\Gamma)$ satisfies
$$
\varkappa_n (\Gamma) \leq c \cdot n
$$
for all $n \geq 1.$}

\vskip2mm

The proof is based on a suitable variation of the approach, going back to A.~Weil, of bounding the dimension of the tangent space to $X_n (\Gamma)$ at a point $[\rho]$ corresponding to a representation $\rho \colon \Gamma \to GL_n (K)$ by the dimension of the cohomology group $H^1 (\Gamma, \mathrm{Ad}_{GL_n} \circ \rho).$ Using the results of \cite{IR}, we describe the latter space in terms of certain spaces of derivations of $R.$  This leads to the conclusion that the constant $c$ in Theorem 2 does not exceed the minimal number of generators $d$ of $R$ (i.e. the smallest integer such that there exists a surjection $\Z[X_1, \dots, X_d] \twoheadrightarrow R$). In fact, if $R$ is the ring of integers or $S$-integers in a number field $L$, then $c = 0$ (see Lemma \ref{D:L-Der}), so we obtain that
$\varkappa_n (\Gamma) = 0$ for all $n$, i.e. $\Gamma$ is $SS$-rigid. We then show in \S \ref{S:SR} that the results of \cite{IR} actually imply that $\Gamma = E(\Phi, R)$ is in fact superrigid in this case.
The proof of Theorem 2 uses the validity of property (T) for $\Gamma = E(\Phi, R).$ On the other hand, groups of this form account for most of the known examples of linear Kazhdan groups, so it is natural to ask if the conclusion of Theorem 2 can be extended to all discrete linear Kazhdan groups.

\vskip2mm

\noindent {\bf Conjecture.} {\it Let $\Gamma$ be a discrete linear group having Kazhdan's property (T). Then there exists a constant $c = c(\Gamma)$ such that $$\varkappa_n (\Gamma) \leq c \cdot n$$ for all $n \geq 1.$}

\vskip2mm

The paper is organized as follows. In \S \ref{S:KT}, we begin by summarizing some well-known facts from $K$-theory and then use the results of \cite{BR} to obtain a description of the group $K_2$ of certain associative rings similar to the one given by Stein in the commutative case. This is then used in the proof of Theorem 1, which is given in \S \ref{S:NC}, along with similar results for arbitrary associative rings.
Next, we begin \S \ref{S:D} with a brief review of representation and character varieties and some related cohomological machinery, after which we turn to the proof of Theorem 2. Finally, in \S \ref{S:SR}, we show how the techniques of \cite{IR}, along with some considerations involving derivations, can be used to establish various rigidity results for the elementary groups $E(\Phi, \I),$ where $\I$ is a ring of algebraic integers.

\vskip2mm

\noindent {\bf Notations and conventions.} Throughout the paper, $\Phi$ will denote a reduced irreducible root system of rank $\geq 2.$ All of our rings are assumed to be associative and unital. As noted earlier, if $R$ is a commutative ring, we say that the pair $(\Phi, R)$ is {\it nice} if $2 \in R^{\times}$ whenever $\Phi$ contains a subsystem of type $B_2$, and $2, 3 \in R^{\times}$ if $\Phi$ is of type $G_2.$ Finally, given an algebraic group $H$, we let $H^{\circ}$ denote the connected component of the identity.

\vskip2mm

\noindent {\bf Acknowledgments.} I would like to thank my advisor Professor G.A.~Margulis for suggesting the problems and for a number of helpful discussions. I am indebted to Professor G.~Prasad for insightful conversations about pseudo-reductive groups during my visit to the University of Michigan. Finally, I thank M.~Kassabov for useful communications regarding algebraic rings.  

\vskip3mm

\section{$K$-theoretic preliminaries} \label{S:KT}

In this section, we develop the $K$-theoretic results that will be needed in the proof of Theorem 1. Even though the statements in this section are consequences of some well-known results, to the best of our knowledge, they have never appeared explicitly in the literature in the form that we require.
The main objective
will be to
use the computations of Bak and Rehmann \cite{BR} to establish certain analogs in the noncommutative setting of Stein's \cite{St2} description of the group $K_2$ of a semilocal commutative ring (see Propositions \ref{KT:P-BR2} and \ref{KT:P-DA} below).

We begin by recalling some standard definitions.
Let $R$ be an associative unital ring. For $1 \leq i,j \leq n, i \neq j,$ and $r \in R$, let $e_{ij}(r) \in GL_n (R)$ be the elementary matrix with $r$ in the $(i,j)$-th place, and denote by $E_n (R)$ the subgroup of $GL_n (R)$, called the {\it elementary group}, generated by all the $e_{ij} (r)$.
If $n \geq 3,$ it is well-known that the elementary matrices in $GL_n (R)$ satisfy the following relations:

\vskip2mm

\noindent (R1) $e_{ij} (r) e_{ij} (s) = e_{ij} (r+s)$;

\vskip2mm

\noindent (R2) $[e_{ij}(r), e_{kl}(s)] =1$ if $i \neq l$, $j \neq k$;

\vskip2mm

\noindent (R3) $[e_{ij} (r), e_{jl} (s)] = e_{il} (rs)$ if $i \neq l.$

\vskip2mm

\noindent The {\it Steinberg group} over $R$, denoted $\mathrm{St}_n (R)$, is defined to be the group
generated by
all symbols $x_{ij} (r)$, with $1 \leq i,j \leq n$, $i \neq j,$ and $r \in R$, subject to the natural analogs of the relations (R1)-(R3) written in terms of the $x_{ij}(r).$
From the definition, it is clear that there exists a canonical surjective group homomorphism
$$
\pi_R \colon \mathrm{St}_n (R) \to E_n (R), \ \ \ x_{ij} (r) \mapsto e_{ij} (r),
$$
and we set
$$
K_2 (n, R) = \ker (\mathrm{St}_n (R) \stackrel{\pi_R}{\longrightarrow} E_n (R)).
$$
It is easy to see that there exist natural homomorphisms
$\mathrm{St}_n (R) \to \mathrm{St}_{n+1} (R)$ and $E_n(R) \hookrightarrow E_{n+1}(R)$, which induce homomorphisms $K_2 (n, R) \to K_2 (n+1, R)$ (cf. \cite{HO}, \S 1.4).
Also notice that the
pair $(\mathrm{St}_n (R), \pi_R)$ is functorial in the following sense: given a homomorphism of rings $f \colon R \to S$, there is a commutative diagram of group homomorphisms
$$
\xymatrix{\mathrm{St}_n (R) \ar[d]_{\pi_R} \ar[r]^{\tilde{F}} & \mathrm{St}_n (S) \ar[d]^{\pi_S} \\ E_n (R) \ar[r]^F & E_n (S) }
$$
where $F$ and $\tilde{F}$ are the homomorphisms induced by $f$, defined on generators by
$$
F \colon e_{ij} (t) \mapsto e_{ij} (f(t)) \ \ \ \mathrm{and} \ \ \  \tilde{F} \colon x_{ij} (t) \mapsto x_{ij} (f(t)).
$$
It follows from the commutativity of the above diagram that $\tilde{F}$ induces a homomorphism $K_2 (n,R) \to K_2 (n,S).$
In the following proposition, we derive some general properties of $K_2 (n,R)$ that will be needed later in this section.

\begin{prop}\label{KT:P-1} \

\noindent {\rm (a)} \parbox[t]{16cm}{Suppose $R$ is an associative unital ring such that $R/\mathrm{Rad}(R)$ is artinian, where $\mathrm{Rad}(R)$ is the Jacobson radical of $R$. Then the natural map $K_2 (3,R) \to K_2 (4,R)$ is an isomorphism. If, moreover, $R$ is finitely generated as a module over its center, then
$K_2 (n,R)$ is a central subgroup of $\mathrm{St}_n (R)$ for $n \geq 3.$}

\vskip1mm

\noindent {\rm (b)} \parbox[t]{16cm}{Suppose $C$ is a commutative finite dimensional algebra over a field $K$ and let $A = M_m (C)$ be the ring of $m \times m$ matrices over $C$. For $a \in C$ and $1 \leq k, l \leq m$, let $\tilde{y}_{kl}(a) \in A$ be the matrix with $a$ as the $(k,l)$-entry and 0 for all other entries.
Then for $n \geq 3$, the maps
$$
\tilde{\psi} (x^A_{ij} (\tilde{y}_{kl}(a))) = x^C_{(i-1)m+k, (j-1)m + l} (a) \ \ \ \ and \ \ \ \ \psi (e^A_{ij} (\tilde{y}_{kl}(a))) = e^C_{(i-1)m+k, (j-1)m + l} (a),
$$
where the $x^A_{ij} (a)$ $($resp. $e^A_{ij} (a))$ are the generators of $\mathrm{St}_n (A)$ $($resp.  $E_n (A))$ and the $x^C_{ij} (c)$ $($resp. $e^C_{ij} (c))$ are the generators of $\mathrm{St}_{nm} (C)$ $($resp. $E_{nm} (C))$, define
isomorphisms $\tilde{\psi} \colon St_n (A) \to St_{nm} (C)$ and $\psi \colon E_n (A) \to E_{nm} (C)$ such that the following diagram commutes:
\begin{equation}\label{NC:D-1}
\xymatrix{St_n (A) \ar[d]_{\pi_{A}} \ar[r]^{\tilde{\psi}} & St_{nm} (C) \ar[d]^{\pi_C} \\ E_{n} (A) \ar[r]^{\psi} & E_{nm} (C)}
\end{equation}
In particular, $K_2 (n, M_m(C)) \simeq K_2 (nm, C).$}
\end{prop}
\begin{proof}
(a) By (\cite{WDK}, Theorem 7), the fact that $R/\mathrm{Rad}(R)$ is artinian implies that it has property $SR_2^*$, and then (\cite{WDK}, Theorem 6) yields the required isomorphism. Now, if $R$ is finitely generated as a module over its center, then according to (\cite{HO}, Theorem 1.4.15), $\pi_R \colon \mathrm{St}_n (R) \to E_n (R)$ is a central extension for $n \geq 4$ (in fact, a {\it universal} central extension for $n \geq 5$). So, in view of the canonical isomorphism $K_2 (3, R) \simeq K_2 (4,R)$,
we obtain that $K_2 (n,R)$ is a central subgroup of $\mathrm{St}_n (R)$ for $n \geq 3,$ as claimed.

\vskip1mm

\noindent (b) First note that since $A$ is generated additively by the $\tilde{y}_{kl} (a)$, with $1 \leq k, l \leq m$, it follows that the $x^A_{ij} (\tilde{y}_{kl}(a))$ and $e^A_{ij} (\tilde{y}_{kl} (a))$ generate $\mathrm{St}_n (A)$ and $E_n (A)$, respectively, so it suffices to define $\tilde{\psi}$ and $\psi$ on these elements.
By direct computation, one verifies that the maps $\tilde{\psi}$ and $\psi$ given in the statement of the proposition are group homomorphisms.
In fact, it is easy to see that $\psi$ is actually an isomorphism for all rings $C$ and $n \geq 3.$

Next, since without loss of generality $m \geq 2$, we have $nm \geq 6,$ so as noted in the proof of (a), $\pi_C \colon St_{nm} (C) \to E_{nm} (C)$ is a {\it universal} central extension and $\pi_A \colon \mathrm{St}_n (A) \to E_n (A)$ is a central extension.
Hence,
there exists a unique group homomorphism $\tilde{\varphi} \colon St_{nm} (C) \to St_n (A)$ making the diagram
\begin{equation}\label{NC:D-2}
\xymatrix{St_n (A) \ar[d]_{\pi_{A}} & \ar[l]_{\tilde{\varphi}}  St_{nm} (C) \ar[d]^{\pi_C} \\ E_{n} (A) \ar[r]^{\psi} & E_{nm} (C)}
\end{equation}
commute, and by universality, we conclude that $\tilde{\psi} \circ \tilde{\varphi} = {\rm id}_{St_{nm} (C)}.$ On the other hand, by the commutativity of the diagrams (\ref{NC:D-1}) and (\ref{NC:D-2}), we have that for any $x \in \mathrm{St}_n (A)$,
$$
(\psi \circ \pi_A \circ \tilde{\varphi} \circ \tilde{\psi})(x) = (\pi_C \circ \tilde{\psi})(x) = (\psi \circ \pi_A)(x). $$
Since $\psi$ is an isomorphism, we conclude that $(\tilde{\varphi} \circ \tilde{\psi})(x) = x z_x,$ where $z_x \in K_2 (n,A).$ The centrality of $K_2 (n,A)$ then implies that the map $x \mapsto z_x$ is a homomorphism $\mathrm{St}_n (A) \to K_2 (n,A)$, which must be trivial as $\mathrm{St}_n (A)$ is a perfect group. Thus, $\tilde{\phi} \circ \tilde{\psi} = {\rm id}_{St_n(A)},$ as required. It immediately follows that $K_2 (n, A) \simeq K_2 (nm, C).$
\end{proof}

Next, let us summarize the results of \cite{BR} dealing with relative $K_2$ groups of associative rings (see Theorem \ref{KT:T-BR1} below). From now on, we will always assume that $n \geq 3.$ First, we need to introduce some additional notation.
As above, let $R$ be an associative unital ring. Given $u \in R^{\times}$, we define, for $i \neq j$, the following standard elements of $\mathrm{St}_n (R)$:
$$
w_{ij} (u) = x_{ij} (u) x_{ji}(-u^{-1}) x_{ij} (u) \ \ \ \mathrm{and} \ \ \ h_{ij} (u) = w_{ij}(u) w_{ij}(-1).
$$
Notice that the image $\pi_R(h_{ij}(u))$ in $E_n (R)$ is the diagonal matrix with $u$ as the $i$th diagonal entry, $u^{-1}$ as the $j$th diagonal entry, and $1$'s everywhere else on the diagonal. We will also need the following noncommutative version of the usual Steinberg symbols: for $u, v \in R^{\times}$, let
$$
c(u,v) = h_{12}(u) h_{12}(v) h_{12}(vu)^{-1}.
$$
One easily sees that $\pi_R (c(u,v))$ is the diagonal matrix with $uvu^{-1} v^{-1}$ as its first diagonal entry and $1$'s everywhere else on the diagonal. Let $U_n (R)$ be the subgroup of $\mathrm{St}_n(R)$ generated by all the $c(u,v)$, with $u,v \in R^{\times}.$

As in the commutative case, one can also consider relative versions of these constructions.
Let $\fa$ be a two-sided ideal of $R$ and
$$
GL_n (R, \fa) := \ker (GL_n (R) \to GL_n (R/ \fa))
$$
be the congruence subgroup of level $\fa$.
Define $E_n (R, \fa)$ to be the normal subgroup of $E_n (R)$ generated by all elementary matrices $e_{ij} (a)$, with $a \in \fa.$ Now letting
$$
\mathrm{St}_n (R, \fa) = \ker (\mathrm{St}_n (R) \to \mathrm{St}_n (R / \fa)),
$$
we have a natural homomorphism $\mathrm{St}_n (R, \fa) \to E_n (R, \fa)$, and we set
$$
K_2 (n, R, \fa) = \ker (\mathrm{St}_n (R, \fa) \to E_n (R, \fa)).
$$
Finally, let
$$
U_n (R, \fa) := < c(u, 1+a) \mid u \in R^{\times}, 1+a \in (1 + \fa) \cap R^{\times}>
$$
(notice this is contained in $\mathrm{St}_n (R, \fa)$). We should point out that even though for a noncommutative ring, the groups $U_n (R)$ and $U_n (R, \fa)$ may not lie in $K_2(n,R)$, it is well-known that any element of $K_2(n,R) \cap U_n (R)$ is automatically contained in the center of $\mathrm{St}_n (R)$ (cf. \cite{M}, Corollary 9.3). This will be needed in Proposition \ref{KT:P-BR2} below.

The following theorem contains a summary of the results of \cite{BR} that will relevant for our purposes.

\begin{thm}\label{KT:T-BR1}{\rm(cf. \cite{BR}, Theorem 2.9 and Corollary 2.11)}
Let $R$ be an associative unital ring. Suppose that $\fa$ is a two sided ideal contained in the Jacobson radical $\mathrm{Rad}(R)$ of $R$, and that $R$ is additively generated by $R^{\times}.$ Assume $n \geq 3.$ Then

\vskip1mm

\noindent {\rm (1)} \parbox[t]{16cm}{$K_2 (n, R, \fa) \subset U_n (R, \fa)$
and the canonical sequence below is exact
$$
1 \to U_n (R, \fa) \to U_n (R) \to U_n (R/ \fa) \to 1.
$$}

\vskip1mm

\noindent {\rm (2)} \parbox[t]{16cm}{If, moreover, $K_2 (n, R/ \fa) \subset U_n (R/ \fa),$ then $K_2 (n,R) \subset U_n (R)$ and the natural sequence
$$
1 \to K_2 (n, R, \fa) \to K_2(n, R) \to K_2 (n, R/ \fa) \to 1
$$
is exact.}
\end{thm}
\noindent The theorem yields the following
\begin{prop}\label{KT:P-BR2}
Suppose that $R$ is either a finite-dimensional algebra over an algebraically closed field $K$ or a finite ring with $2 \in R^{\times}.$ Then $K_2 (n, R) \subset U_n(R)$, and consequently $K_2 (n,R)$ is a central subgroup of $\mathrm{St}_n (R).$
\end{prop}
\begin{proof}
Let $J = \mathrm{Rad}(R)$ be the Jacobson radical of $R$. To apply Theorem \ref{KT:T-BR1}, we need to verify that in both cases,
$R$ is additively generated by its units and that
$K_2 (n, R/J) \subset U_n (R/J).$

If $R$ is a finite-dimensional algebra over $K$, then we can view $R$ as a connected algebraic ring over $K$, and it follows from (\cite{IR}, Corollary 2.5) that $R$ is generated by $R^{\times}$.\footnotemark \footnotetext{All of the background on algebraic rings needed in this paper can be found in \cite{IR}, \S2. M.~Kassabov has also informed us that the notion of an algebraic ring actually goes back to Greenberg's paper \cite{Gr}, where one can find proofs of some basic properties.} Now suppose that $R$ is a finite ring. Since $R$ is obviously artinian,
$R/J$ is semisimple (\cite{L}, Theorem 4.14), so by the Artin-Wedderburn Theorem (\cite{L}, Theorem 3.5) and the fact that finite division rings are commutative (\cite{L}, Theorem 13.1), we have
$$
R/J \simeq M_{n_1} (F_1) \oplus \cdots \oplus M_{n_r} (F_r),
$$
where $F_1, \dots, F_r$ are finite fields, with $F_i \neq \mathbb{F}_2$, the field of two elements, for all $i$ as $2 \in R^{\times}.$
It follows that $R/J$ is additively generated by its units. On the other hand, the canonical map $R \to R/J$ induces a surjective homomorphism $R^{\times} \to (R/J)^{\times},$ which, combined with the
the fact that $J$ lies in the linear span of $R^{\times}$ (cf. \cite{L}, Lemma 4.3), yields that $R$ is additively generated by $R^{\times}.$

Next, let us show that $K_2 (n, R/J) \subset U_n (R/J)$ in both cases. If $R$ is a finite-dimensional $K$-algebra, then as above, $R/J$ is semisimple. So, since there are no nontrivial division algebras over algebraically closed fields, the Artin-Wedderburn Theorem implies that
$$
R/J \simeq M_{n_1} (K) \oplus \cdots \oplus M_{n_s} (K).
$$
Thus, in both cases, $R/J$ is a direct sum of matrix algebras over fields.
Since $K_2$ commutes with finite direct sums, we may assume without loss of generality that $A := R/J \simeq M_m (F)$, with $F$ a field. By Proposition \ref{KT:P-1}, we have isomorphisms $\tilde{\psi} \colon \mathrm{St}_n (A) \to \mathrm{St}_{nm} (F)$ and $\psi \colon E_n (A) \to E_{nm} (F)$ that induce an isomorphism $K_2 (n, A) \simeq K_2 (nm, F).$ Now let $u \in F^{\times}$ and $t_u = \mathrm{diag}(u, 1, \dots, 1) \in M_m (F).$ By direct computation, one checks that
$$
\tilde{\psi} (h_{12}^A (t_u)) = h_{1, m+1}^F (u),
$$
and therefore, for $u, v \in F^{\times},$ we have
$$
\tilde{\psi} (c(t_u, t_v)) = c_{1, m+1} (u,v),
$$
where $c_{1, m+1} (u,v) = h_{1, m+1}^F (u) h_{1,m+1}^F (v) h_{1, m+1}^F (vu)^{-1}.$ On the other hand, by Matsumoto's theorem,
the group $K_2 (nm, F)$ is generated by the Steinberg symbols $c_{1, m+1} (u,v)$ (cf. \cite{Stb1}); consequently, we see that
$K_2 (n, R/J) \subset U_n (R/J)$, as claimed. Hence, $K_2 (n, R) \subset U_n (R)$
by Theorem \ref{KT:T-BR1}. As we noted above, it now follows from (\cite{M}, Corollary 9.3) that $K_2 (n,R)$ lies in the center of $\mathrm{St}_n (R).$

\end{proof}

An important ingredient in the proof of Theorem 1 will be the following proposition.


\begin{prop}\label{KT:P-DA}
Let $k$ and $K$ be fields of characteristic 0, with $K$ algebraically closed. Suppose that $D$ is a finite-dimensional central division algebra over $k,$ $A$ a finite-dimensional algebra over $K$, and $f \colon D \to A$ a ring homomorphism with Zariski-dense image. Then for $n \geq 3,$ $K_2 (n,A)$ coincides with the subgroup
$$
U'_n (A) = < c(u,v) \mid u,v \in \overline{f(L^{\times})} >,
$$
where $L$ is an arbitrary maximal subfield of $D$.
\end{prop}

We begin with

\begin{lemma}\label{KT:L-DA}
Let $A,$ $D$, and $f$ be as above, and set
$C = \overline{f(k)}$ (Zariski closure). Then
\begin{equation}\label{E-LDA}
A \simeq D \otimes_k C \simeq M_{s} (C)
\end{equation} 
as $K$-algebras, where $s^2 = \dim_k D.$ Moreover, if $L$ is any maximal subfield of $D$, then the second isomorphism can be chosen so that $L \otimes_k C \simeq D_s (C)$, where $D_s (C) \subset M_s (C)$ is the subring of diagonal matrices.
\end{lemma}
\begin{proof}
We start with the proof of the first isomorphism in (\ref{E-LDA}).
To begin, we note that since $k$ and $K$ are both fields of characteristic 0,
$C$ is a finite-dimensional algebra over $K$ by (\cite{IR}, Lemma 2.13 and Proposition 2.14). Moreover, by (\cite{Gr}, Proposition 5.1), the natural inclusion $C \hookrightarrow A$ is a homomorphism of $K$-algebras (this also follows from the proof of \cite{IR}, Proposition 2.14).
Now consider the map
$$
\theta \colon D \otimes_k C \to A,  \ \ \ (x,c) \mapsto f(x) c.
$$
We claim that $\theta$ is an isomorphism. From the above remark, it is clear that $\theta$ is a homomorphism of $K$-algebras (where $D \otimes_k C$ is endowed with the natural $K$-algebra structure coming from $C$).
For surjectivity, first note that since $\mathrm{im} \ \theta$ contains $f(D)$, it is Zariski-dense in $A$.
On the other hand, let $x_1, \dots, x_{s^2}$ be a basis of $D$ over $k.$ Then
$$
\mathrm{im} \ \theta = f(x_1) C + \dots + f(x_{s^2}) C,
$$
and therefore is closed. Hence, $\theta$ is surjective. To prove injectivity, notice that
since
$D$ is a central simple algebra, $\ker \theta = D \otimes_k \mathfrak{c}$ for some ideal $\mathfrak{c} \subset C$ (see \cite{DF}, Theorem 3.5). On the other hand, $f(1_D) = 1_A$ (as $f$ is a ring homomorphism), so $\mathfrak{c} = 0$, and $\theta$ is injective.

Now let us consider the second isomorphism. First, since $C$ is a commutative artinian algebraic ring, by (\cite{IR}, Proposition 2.20) we can write
$$
C = C_1 \times \cdots \times C_r,
$$
where each $C_i$ is a local commutative algebraic ring. Moreover, since tensor products commute with finite products and $M_s (C_1 \times \cdots \times C_r) = M_s (C_1) \times \cdots \times M_s (C_r)$, it suffices to establish the isomorphism when $C$ is a local algebraic ring. So, suppose that is the case and let $J(C)$ be the Jacobson radical of $C$. Then it follows from (\cite{IR}, Corollary 2.6 and Proposition 2.19) that $C/J(C) \simeq K$, so composing $f$ with the canonical map $C \to C/J(C)$, we obtain an embedding $k \hookrightarrow K.$ Consequently, as $K$ is algebraically closed, the division algebra $D$ splits over $K$, i.e. there exists an isomorphism
\begin{equation}\label{E-DA}
\tau \colon D \otimes_k K \stackrel{\sim}{\longrightarrow} M_s (K).
\end{equation}
Notice also that if $L$ is a maximal subfield of $D$, we can choose $\tau$ so that $L \otimes_k K \simeq D_s (K).$ Indeed, since $L$ is separable over $k$ (as char $k = 0$) and $[L:k] = s,$ we can write $L = k[X]/(f)$, where $f$ is a separable polynomial of degree $s$. Then by the Chinese Remainder Theorem, $L \otimes_k K \simeq K^s.$ But any subalgebra of $M_s (K)$ which is isomorphic to $K^s$ is conjugate to $D_s (K)$ (\cite{GS}, Lemma 2.2.9), so it follows that $\tau$ can be composed with an inner automorphism of $M_s (K)$ to have the required form.

Now consider the natural (surjective) map
$$
D \otimes_k C \to D \otimes_k (C/J(C)) = D \otimes_k K.
$$
Since $D$ is a central simple algebra, the same argument as above shows that the kernel of this map is contained in the Jacobson radical $J(D \otimes_k C)$, and the fact that $D \otimes_k K \simeq M_s (K)$ is semisimple implies that it actually coincides with $J(D \otimes_k C).$ So, by the Wedderburn-Malcev theorem (cf. \cite{P}, Corollary 11.6), there exists a section
$$
\alpha \colon M_s (K) \simeq D \otimes_k K \hookrightarrow D \otimes_k C.
$$
We claim that the map
$$
\beta \colon M_s (K) \otimes_K C \to D \otimes_k C, \ \ \ m \otimes c \mapsto \alpha(m) \cdot (1 \otimes c)
$$
gives the required isomorphism. Indeed, injectivity is proved by the same argument as above, and surjectivity follows by dimension count. Thus, $M_s (C) \simeq M_s (K) \otimes_K C \simeq D \otimes_k C,$ and it follows immediately from the above remarks that $D_s (C) \simeq L \otimes_k C.$

\end{proof}

\vskip2mm

\noindent {\it Proof of Proposition \ref{KT:P-DA}}.
By Lemma \ref{KT:L-DA}, we have
$L \otimes_k C \simeq D_s (C).$ Moreover,
$L \otimes_k C \simeq \overline{f(L)}$. Indeed, since $k \subset L$, we have
$$
f(L) \subset \theta (L \otimes_k C) \subset \overline{f(L)}.
$$
On the other hand, the same argument as in the proof of Lemma \ref{KT:L-DA}
shows that $\theta (L \otimes_k C)$ is closed.

Next, since $A \simeq M_s (C)$ and $C$ is a finite-dimensional $K$-algebra, by Proposition \ref{KT:P-1} there exists an isomorphism $\tilde{\psi} \colon \mathrm{St}_n(A) \to \mathrm{St}_{ns}(C)$ that induces an isomorphism $K_2 (n,A) \simeq K_2 (ns, C).$ Now, $C$ is a semilocal commutative ring which is additively generated by its units, so by (\cite{St2}, Theorem 2.13), $K_2 (ns, C)$
coincides with the subgroup $U_{ns} (C)$ of $\mathrm{St}_{ns}(C)$
generated by the Steinberg symbols $c_{1,s+1}(u,v)$ taken with respect to the root $\alpha_{1,s+1}$ (i.e. $c_{1,s+1} (u,v) = h_{1,s+1} (u) h_{1,s+1}(v) h_{1,s+1} (vu)^{-1}$). As we noted in the proof of Proposition \ref{KT:P-BR2}, we have
$$
\tilde{\psi}(c(t_u, t_v)) = c_{1, s+1}(u,v),
$$
where for $u \in C^{\times}$, we set $t_u = {\rm diag}(u,1, \dots, 1) \in M_s (C).$
Thus, $K_2 (n, A)$ is contained in the group generated by the symbols $c(t_u, t_v)$. On the other hand, since all of the $t_u$ diagonal matrices, they lie in the image of $L \otimes_k C$, hence $K_2 (n,A) \subset U'_n (A).$
Since clearly $U'_n (A) \subset K_2 (n,A)$, this concludes the proof.
\hfill $\Box$

\section{Abstract homomorphisms over non-commutative rings}\label{S:NC}

The main goal of this section is to give the proof of Theorem 1.
Before beginning the argument, we would like to give an alternative statement of Theorem 1, which can be generalized (in a somewhat weaker form) to (essentially) arbitrary associative rings. 
First, we need to observe that if $B$ is a finite-dimensional algebra over an algebraically closed field $K$, then the elementary group $E_n (B)$ has the structure of a connected algebraic $K$-group. Indeed, using the regular representation of $B$ over $K$, it is easy to see that $GL_n (B)$ is a Zariski-open subset of $M_n (B),$ and hence an algebraic group over $K$. Now let us
view $B$ as a connected algebraic ring over $K$, and for $i, j \in \{1, \dots, n \},$ $i \neq j$, consider the regular maps
$$
\varphi_{ij} \colon B \to GL_n (B), \ \ \ b \mapsto e_{ij} (b)
$$
Set $W_{ij} = \mathrm{im} \ \varphi_{ij}.$ Then each $W_{ij}$ contains the identity matrix $I_n \in GL_n (B)$, and by definition $E_n (B)$ is generated by the $W_{ij}.$ So, $E_n (B)$ is a connected algebraic group by (\cite{Bo}, Proposition 2.2).

\begin{thm}\label{NC:T-NonComm1}
Suppose $k$ and $K$ are fields of characteristic 0, with $K$ algebraically closed, $D$ is a finite-dimensional central division algebra over $k$, and $n$ is an integer $\geq 3.$ Let $\rho \colon E_n (D) \to GL_m (K)$ be a finite-dimensional linear representation and set $H = \overline{\rho (E_n (D))}$ (Zariski closure). Then there exists a finite-dimensional associative $K$-algebra $\B$, a ring homomorphism $f \colon D \to \B$ with Zariski-dense image, and a morphism $\sigma \colon E_n (\B) \to H$ of algebraic $K$-groups such that
$$
\rho = \sigma \circ F,
$$
where $F \colon E_n (D) \to E_n (\B)$ is the group homomorphism induced by $f.$
\end{thm}

We also have the following result for general associative rings.

\begin{thm}\label{NC:T-NonComm2}
Suppose $R$ is an associative ring with $2 \in R^{\times},$ $K$ is an algebraically closed field of characteristic 0, and $n$ is an integer $\geq 3.$ Let $\rho \colon E_n (R) \to GL_m (K)$ be a finite-dimensional linear representation, set $H = \overline{\rho (E_n (R))}$, and denote by $H^{\circ}$ the connected component of $H$. If the unipotent radical of $H^{\circ}$ is commutative, there exists a finite-dimensional associative $K$-algebra $\B$, a ring homomorphism $f \colon R \to \B$ with Zariski-dense image, and a morphism $\sigma \colon E_n (\B) \to H$ of algebraic $K$-groups such that for a suitable finite-index subgroup $\Delta \subset E_n (R)$, we have
$$
\rho \vert_{\Delta} = (\sigma \circ F) \vert_{\Delta},
$$
where $F \colon E_n (R) \to E_n (\B)$ is the group homomorphism induced by $f.$
\end{thm}

As we indicated in the introduction, the proofs of Theorems \ref{NC:T-NonComm1} and \ref{NC:T-NonComm2}
are based on a natural extension of the approach developed in our earlier paper \cite{IR}.
More precisely, we will first associate to $\rho$ an algebraic ring $A$, then show that $\rho$ can be lifted to a representation $\tilde{\tau} \colon \mathrm{St}_n (A) \to H$ of the Steinberg group, and finally use the results of \S \ref{S:KT} to verify that $\tilde{\sigma}$ descends to an abstract representation of $E_n (A)$. Then, to conclude the argument, we will prove that this abstract representation is actually a morphism of algebraic groups.

We begin with the construction of the algebraic ring $A$ attached to a given representation $\rho$.
\begin{prop}\label{NC:P-AR}
Suppose $R$ is an associative ring, $K$ an algebraically closed field, and $n \geq 3$. Given a representation $\rho \colon E_n (R) \to GL_m (K)$, there exists an associative algebraic ring $A$, together with a homomorphism of abstract rings $f \colon R \to A$ having Zariski-dense image such that for all $i, j \in \{ 1, \dots, n \},$ $i \neq j$, there is an injective regular map $\psi_{ij} \colon A \to H$ into $H := \overline{\rho (E_n (R))}$ satisfying
\begin{equation}\label{E:AR-1}
\rho (e_{ij} (t)) = \psi_{ij} (f(t))
\end{equation}
for all $t \in R.$
\end{prop}
\begin{proof}
This statement goes back to \cite{Kas} (see also \cite{IR}, Theorem 3.1). For the sake of completeness, we indicate
the main points of the construction. Let $A = \overline{\rho (e_{13} (R))}$. If $\ba \colon A \times A \to A$ denotes the restriction of the matrix product in $H$ to $A$, it is clear $(A, \ba)$ is a commutative algebraic subgroup of $H$. We let $f \colon R \to A$ be the map defined by $t \mapsto \rho(e_{13} (t))$. From the definition, it follows that
$$
\ba (f(t_1), f(t_2)) = f(t_1 + t_2)
$$
for all $t_1, t_2 \in R.$ To define the multiplication operation $\bm \colon A \times A \to A$, we will need the following elements:
$$
\overline{w}_{12} = e_{12} (1) e_{21} (-1) e_{12} (1) \ \ \ \mathrm{and} \ \ \ \overline{w}_{23} = e_{23} (1) e_{32}(-1) e_{23} (1)
$$
(notice that these are simply the images under $\pi_{R}$ of the elements $w_{ij} (1)$ considered in \S \ref{S:KT}).
By direct computation, one sees that
$$
\overline{w}_{12}^{-1} e_{13}(r) \overline{w}_{12} = e_{23} (r), \ \ \ \overline{w}_{23} e_{13} (r) \overline{w}_{23}^{-1} = e_{12} (r)
$$
and
$$
[e_{12} (r), e_{23} (s)] = e_{13} (rs)
$$
for all $r, s \in R$, where $[g, h] = g h g^{-1} h^{-1}.$ Now let $\bm \colon A \times A \to H$  be the regular map defined by
$$
\bm (a_1, a_2) = [\rho(\overline{w}_{23}) a_1 \rho(\overline{w}_{23})^{-1}, \rho(\overline{w}_{12})^{-1} a_2 \rho(\overline{w}_{12})].
$$
Then the above relations yield
$$
\bm (f(t_1), f(t_2)) = f(t_1 t_2),
$$
so, in particular, $\bm (f(R) \times f(R)) \subset f(R),$ which implies that
$\bm (A \times A) \subset A$ and allows us to view $\bm$ as a regular map $\bm \colon A \times A \to A.$ Since by our assumption $R$ is a (unital) associative ring and $f$ has Zariski-dense image, it follows that
$(A, \ba, \bm)$ is a (unital) associative algebraic ring, as defined in \cite{IR}, \S 2. Furthermore, by our construction, (\ref{E:AR-1}) obviously holds for the inclusion map $\psi_{13} \colon A \to H.$ Finally, using an appropriate element $\overline{w}_{ij},$ we can conjugate any root subgroup $e_{ij} (R)$ into $e_{13} (R),$ from which the
existence of all the other maps $\psi_{ij}$ follows.
\end{proof}

\vskip2mm

\noindent {\bf Remark 3.4.} Observe that if $R$ is an infinite division ring,
then the algebraic ring $A$ constructed in Proposition \ref{NC:P-AR} is automatically connected. Indeed, the
connected component $A^{\circ}$ is easily seen to be a two-sided ideal of $A$. So, if $A \neq A^{\circ},$ then $f^{-1} (A^{\circ})$ would be a proper two-sided ideal of finite index in $R$, which is impossible. In particular, we see that in the situation of Theorem \ref{NC:T-NonComm1}, the algebraic ring associated to $\rho$ is connected.

\vskip2mm

Next, we show that the representation $\rho$ can be lifted to a representation of the Steinberg group $\mathrm{St}_n (A).$ The precise statement is given by the following proposition.

\addtocounter{thm}{1}

\begin{prop}\label{NC:P-St1}
Suppose $R$ is an associative ring, $K$ an algebraically closed field, and $n \geq 3$, and let
$\rho \colon E_n (R) \to GL_m (K)$ a representation. Furthermore, let $A$ and $f \colon R \to A$ be the algebraic ring and ring homomorphism constructed in Proposition \ref{NC:P-AR}.
Then there exists a group homomorphism $\tilde{\tau} \colon \mathrm{St}_n (A) \to H \subset GL_m (K)$ such that $\tilde{\tau} \colon x_{ij} (a) \mapsto \psi_{ij} (a)$ for all $a \in A$ and all $i,j \in \{1, \dots, n \}, i \neq j.$ Consequently, $\tilde{\tau} \circ \tilde{F} = \rho \circ \pi_R,$ where $\tilde{F} \colon \mathrm{St}_n (R) \to \mathrm{St}_n (A)$ is the homomorphism induced by $f.$
\end{prop}
\begin{proof}
This proposition is proved in exactly the same way as (\cite{IR}, Proposition 4.2). We simply note that since
$\mathrm{St}_n (A)$ is generated by the symbols $x_{ij} (a)$ subject to the relations (R1)-(R3) given in \S \ref{S:KT}, to establish the existence of $\tilde{\tau},$ it suffices to verify that
relations (R1)-(R3) are satisfied if the $x_{\ij} (a)$ are replaced by $\psi_{ij} (a)$, which follows from (\ref{E:AR-1}) and the fact that $f$ has Zariski-dense image.
For the second statement, we observe that
the maps $\tilde{\tau} \circ \tilde{F}$ and $\rho \circ \pi_R$ both send the symbol $x_{ij} (s)$ to $\psi_{ij} (f(s)) = \rho (e_{ij} (s)) = (\rho \circ \pi_R)(x_{ij} (s)),$ so they must coincide on $\mathrm{St}_n (R).$
\end{proof}

To analyze the representation $\tilde{\sigma}$ that we have just constructed, we will need some additional information on the structure of the group $\mathrm{St}_n (A).$


\begin{prop}\label{NC:P-St2}
Let $K$ be an algebraically closed field of characteristic 0 and $n$ and integer $\geq 3.$ Suppose $A$ is an associative algebraic ring over $K$ such that $2 \in A^{\times}$ and denote by $A^{\circ}$ the connected component of $0_A.$ Then

\vskip1mm

\noindent $\mathrm{(i)}$ $\mathrm{St}_n (A) = \mathrm{St}_n (A^{\circ}) \times P$, where $P$ is a finite group;

\vskip1mm

\noindent $\mathrm{(ii)}$ $K_2 (n, A^{\circ})$ is a central subgroup of $\mathrm{St}_n (A^{\circ}).$
\end{prop}
\begin{proof}
(i) First, since char $K = 0$, by (\cite{IR}, Proposition 2.14), we have $A = A^{\circ} \oplus S$, with $S$ a finite ring.
So,
$$
\mathrm{St}_n (A) = \mathrm{St}_n (A^{\circ}) \times \mathrm{St}_n (S),
$$
and we need to show that $\mathrm{St}_n(S)$ is a finite group. Now, since $E_n (S)$ is obviously a finite group and $K_2 (n, S)$ is by definition the kernel of the canonical map $\pi_S \colon \mathrm{St}_n (S) \to E_n (S)$, we see that the finiteness of $\mathrm{St}_n (S)$ is equivalent to that of $K_2 (n, S).$ On the other hand, since $2 \in S^{\times}$, Proposition \ref{KT:P-BR2} implies that
$K_2 (n,S)$ is a central subgroup of $\mathrm{St}_n (S)$. So, we can use the argument given in the proof of (\cite{IR}, Proposition 4.5) and consider the Hochschild-Serre spectral sequence
$$
H^1 (\mathrm{St}_n (S), \Q/ \Z) \to H^1 (K_2 (\Phi, S), \Q/ \Z)^{\mathrm{St}_n (S)} \to H^2 (E_n (S), \Q/ \Z)
$$
(where all groups act trivially on $\Q/ \Z$) corresponding to the short exact sequence
$$
1 \to K_2 (n, S) \to \mathrm{St}_n(S) \stackrel{\pi_S}{\longrightarrow} E_n (S) \to 1
$$
to conclude that $K_2 (n,S)$ is finite.


\vskip2mm

\noindent (ii) By (\cite{IR}, Proposition 2.14), $A^{\circ}$ is a finite-dimensional $K$-algebra, so
the assertion follows from Proposition \ref{KT:P-BR2}.
\end{proof}

\vskip2mm

\noindent {\bf Remark 3.7} We would like to point out that the assumption that $2 \in A^{\times}$ is needed to guarantee that the finite ring $S$ that appears in the proof of Proposition \ref{NC:P-St2}(i) above is additively generated by its units, which then enables us to use Proposition \ref{KT:P-BR2} to conclude that $K_2 (n,S)$ is a central subgroup of $\mathrm{St}_n (S).$ If $S$ is a finite commutative ring, then, as we show in (\cite{IR}, Proposition 4.5), this assumption is not needed since in that case $S$ can be written as a finite product of commutative local rings, which are automatically generated by their units.

\vskip2mm

To complete the proofs of Theorems \ref{NC:T-NonComm1} and \ref{NC:T-NonComm2}, the basic idea will be to show that
the homomorphism $\tilde{\tau}$ constructed in Proposition \ref{NC:P-St1} descends to a (rational) representation of $E_n (A).$ Let us make this more precise. Given a representation $\rho \colon E_n (R) \to GL_m (K),$ let $f \colon R \to A$ be the ring homomorphism associated to $\rho$ (Proposition \ref{NC:P-AR}), and
denote by
$\tilde{F} \colon \mathrm{St}_n (R) \to \mathrm{St}_n (A)$ and $F \colon E_n (R) \to E_n (A)$ the group homomorphisms induced $f.$
Then under the hypotheses of Theorems \ref{NC:T-NonComm1} and \ref{NC:T-NonComm2}, we have $\mathrm{St}_n (A) = \mathrm{St}_n (A^{\circ})$ (Remark 3.4) and $\mathrm{St}_n (A) = \mathrm{St}_n (A^{\circ}) \times P$ (Proposition \ref{NC:P-St2}), respectively,
so in both cases,
$\tilde{\Delta} := \tilde{F}^{-1} (\mathrm{St}_n (A^{\circ}))$ and $\Delta : = \pi_R (\tilde{\Delta})$ are finite-index subgroups of  $\mathrm{St}_n (R)$ and $E_n (R)$. Moreover, it is clear that
$F(\Delta) \subset E_n (A^{\circ})$.
Thus, letting
$\tilde{\sigma}$ denote the restriction of $\tilde{\tau}$ to $\mathrm{St}_n (A^{\circ})$, we see that the solid arrows in
\begin{equation}\label{NC:D-3}
\xymatrix{\tilde{\Delta} \ar[r]^{\tilde{F}} \ar[d]_{\pi_R} & \mathrm{St}_n (A^{\circ}) \ar[d]^{\pi_{A^{\circ}}} \ar[rrdd]^{\tilde{\sigma}} \\ \Delta \ar[rrrd]_{\rho} \ar[r]^{F} & E_n(A^{\circ}) \ar@{.>}[rrd]^{\sigma} \\ & & & H^{\circ} \\}
\end{equation}
form a commutative diagram. In the remainder of this section, we will show that under our assumptions, there exists
a group homomorphism $\sigma \colon E_n (A^{\circ}) \to H^{\circ}$ (in fact, a morphism of algebraic groups)
making the full diagram commute. In the situation of Theorem \ref{NC:T-NonComm1}, the existence of the required {\it abstract} homomorphism $\sigma$ will be shown in Proposition \ref{NC:P-SL} below. For Theorem \ref{NC:T-NonComm2},
we will first need to establish the somewhat weaker result that
that there exists a homomorphism $\bar{\sigma} \colon E_n (A^{\circ}) \to \bar{H}$ such that $\bar{\sigma} \circ \pi_{A^{\circ}} = \nu \circ \tilde{\sigma}$, where $Z(H^{\circ})$ is the center of $H^{\circ}$, $\bar{H} = H^{\circ}/ Z(H^{\circ})$, and $\nu \colon H^{\circ} \to \bar{H}$ is the canonical map (see Proposition \ref{NC:P-FI})

\begin{prop}\label{NC:P-SL}
Suppose $k$ and $K$ are fields of characteristic 0, with $K$ algebraically closed, $D$ is a finite-dimensional central division algebra over $k$, and $n$ is an integer $\geq 3.$ Let $\rho \colon E_n (D) \to GL_m (K)$ be a representation and denote by $A$ the algebraic ring associated to $\rho$ (Proposition \ref{NC:P-AR}).
Then $A = A^{\circ}$ is a finite-dimensional $K$-algebra and there exists a homomorphism of abstract groups $\sigma \colon E_n (A^{\circ}) \to H^{\circ}$ making the diagram $\mathrm{(\ref{NC:D-3})}$ commute.
\end{prop}
\begin{proof}
We have $A = A^{\circ}$ by Remark 3.4, and $A^{\circ}$ is a finite-dimensional $K$-algebra by (\cite{IR}, Proposition 2.14).
Next, by Proposition \ref{KT:P-DA}, $K_2 (n, A)$ coincides with the subgroup
$$
U'_n (A) = < c(u,v) \mid u,v \in \overline{f(L^{\times})} >,
$$
of $\mathrm{St}_n (A)$,
where $L$ is an arbitrary maximal subfield of $D$ and $f \colon D \to A$ is the ring homomorphism associated to $\rho.$ Now,
from the construction of $\tilde{\sigma}$ and the definition of $c(u,v),$ we have
$$
\tilde{\sigma}(c(u,v)) = H_{12} (u) H_{12} (v) H_{12} (vu)^{-1},
$$
where for $r \in A^{\times},$ we set
$$
H_{12} (r)= W_{12} (r) W_{12} (-1) \ \ \ \mathrm{and} \ \ \ W_{12} (r) = \psi_{12} (r) \psi_{21} (-r^{-1}) \psi_{12} (r).
$$
By (\cite{IR}, Proposition 2.4), the map $A^{\times} \to A^{\times},$ $t \mapsto t^{-1}$ is regular, which implies that the map
$$
\Theta \colon A^{\times} \times A^{\times} \to H, \ \ \ (u,v) \mapsto \tilde{\tau} (c(u,v))
$$
is also regular. On the other hand, as we observed earlier, $\pi_D (h_{ij} (u)) \in E_n (D)$ is a diagonal matrix with $u$ as the $i$th diagonal entry, $u^{-1}$ as the $j$th diagonal entry, and 1's everywhere else on the diagonal. In particular, for $u, v \in L^{\times}$ it follows that
$$
\pi_D (h_{12} (u) h_{12} (v) h_{12} (vu)^{-1}) = 1.
$$
So, by Proposition \ref{NC:P-St1},
$$
\tilde{\sigma} (c(f(u), f(v))) = \rho (\pi_D (h_{12} (u) h_{12} (v) h_{12} (vu)^{-1})) = 1
$$
for all $u, v \in L^{\times}.$ By the regularity of $\Theta,$ we obtain that $\tilde{\sigma} (c(a,b)) = 1$ for all $a, b \in \overline{f(L^{\times})},$ and consequently $\tilde{\sigma}$ vanishes on $K_2 (n,A).$ Since the canonical homomorphism $\pi_A \colon \mathrm{St}_n (A) \to E_n (A)$ is surjective and by definition $K_2 (n,A) = \ker \pi_A,$
the existence of $\sigma$ now follows.
\end{proof}

The proof of Theorem \ref{NC:T-NonComm2} will require the following proposition,
which contains analogs of results established in (\cite{IR}, \S 5).

\begin{prop}\label{NC:P-St3}
Suppose $R$ is an associative ring with $2 \in R^{\times}$, $K$ is an algebraically closed field of characteristic 0,
and $n \geq 3.$ Let $\rho \colon E_n (R) \to GL_m (K)$ be a representation, set $H = \overline{\rho (E_n (R))}$, and denote by $A$ the algebraic ring associated to $\rho.$  Then

\vskip1mm

\noindent $\mathrm{(i)}$ \parbox[t]{16cm}{The group $H^{\circ}$ coincides with $\tilde{\sigma} (\mathrm{St}_n (A^{\circ}))$ and is its own commutator.}

\vskip1mm

\noindent $\mathrm{(ii)}$ \parbox[t]{16cm}{Let $U$ and $Z(H^{\circ})$ be the unipotent radical and center of $H^{\circ},$ respectively.
If $U$ is commutative, then $Z(H^{\circ}) \cap U = \{e \}$, and consequently, $Z(H^{\circ})$ is finite and is contained in any Levi subgroup of $H^{\circ}.$}
\end{prop}

\begin{proof}
(i) It follows from Proposition \ref{NC:P-St1} that
$\tilde{\sigma} (\mathrm{St}_n (A^{\circ}))$ coincides with the (abstract) group $\mathcal{H} \subset H$ generated by all the $\psi_{ij} (A^{\circ})$, with $i, j \in \{ 1, \dots, n \}, i \neq j.$
Since $\psi_{\alpha} (A^{\circ})$ is clearly a connected subgroup of $H$,
by (\cite{Bo}, Proposition 2.2)
$\mathcal{H}$ is Zariski-closed and connected, hence $\mathcal{H} \subset H^{\circ}.$
On the other hand, by Proposition \ref{NC:P-St2}, $\mathrm{St}_n (A^{\circ})$ is a finite-index subgroup of $\mathrm{St}_n (A),$ from which it follows that $\tilde{\sigma} (\mathrm{St}_n (A))$ is Zariski-closed. Since $\tilde{\sigma} (\mathrm{St}_n (A))$ contains $\rho (E_n (R))$, it is Zariski-dense in $H$, and therefore coincides with $H$. So, $\mathcal{H}$ is a closed subgroup of finite index in $H$, hence $\mathcal{H} \supset H^{\circ}$, and consequently $\mathcal{H} = H^{\circ}.$  Furthermore, from the definition of the Steinberg group, one
easily sees that $\mathrm{St}_n (A^{\circ})$ coincides with its commutator subgroup, so the same is true for $H^{\circ}.$

\vskip1mm

\noindent (ii) Using the fact that $H^{\circ}$ coincides with its commutator subgroup, one can now apply the argument given in the proof of (\cite{IR}, Proposition 5.5).

\end{proof}

Now set $\bar{H} = H^{\circ}/ Z(H^{\circ}).$ Since $Z(H^{\circ})$ is a closed normal subgroup of $H^{\circ}$, $\bar{H}$ is an (affine) algebraic group and the canonical map $\nu \colon H^{\circ} \to \bar{H}$ is a morphism of algebraic groups (\cite{Bo}, Theorem 6.8).

\begin{prop}\label{NC:P-FI}
Suppose $R$ is an associative ring with $2 \in R^{\times}$, $K$ is an algebraically closed field of characteristic 0,
and $n \geq 3.$ Let $\rho \colon E_n (R) \to GL_m (K)$ be a representation, set $H = \overline{\rho (E_n (R))}$, and denote by $A$ the algebraic ring associated to $\rho.$  Then $A^{\circ}$ is a finite-dimensional $K$-algebra and
there exists a homomorphism
$\bar{\sigma} \colon E_n (A^{\circ}) \to \bar{H}$ such that $\bar{\sigma} \circ \pi_{A^{\circ}} = \nu \circ \tilde{\sigma}.$
\end{prop}
\begin{proof}
Since char $K = 0,$ by (\cite{IR}, Proposition 2.14) $A^{\circ}$ is a finite-dimensional $K$-algebra.
Furthermore, $H^{\circ} = \tilde{\sigma} (\mathrm{St}_n (A^{\circ}))$ by Proposition \ref{NC:P-St3} and $K_2 (n, A^{\circ}) = \ker \pi_{A^{\circ}}$ is a central subgroup of $\mathrm{St}_n (A^{\circ})$ by Proposition \ref{KT:P-BR2}, from which the existence of $\bar{\sigma}$ follows.
\end{proof}

The remaining step in the proof is to show
that the (abstract) homomorphisms $\sigma \colon E_n (A^{\circ}) \to H^{\circ}$ and $\bar{\sigma} \colon E_n (A^{\circ}) \to \bar{H}$ constructed in Propositions \ref{NC:P-SL} and \ref{NC:P-FI}, respectively, are actually morphisms of algebraic groups (see Proposition \ref{NC:P-R2} below). In the latter case, this will allow us to lift $\bar{\sigma}$
to a
morphism of algebraic groups $\sigma \colon E_n (A^{\circ}) \to H^{\circ}$ making the diagram (\ref{NC:D-3}) commute.
Our proof of rationality here will deviate from the approach of \cite{IR}, as rather than using results about the ``big cell" of $E_n (A^{\circ})$, we will instead apply the following geometric lemma.


\begin{lemma}\label{NC:L-R1}
Let $X, Y, Z$ be irreducible varieties over an algebraically closed field $K$ of characteristic 0. Suppose $s \colon X \to Y$ and $t \colon X \to Z$ are regular maps, with $s$ dominant, such that for any $x_1, x_2 \in X$ with $s(x_1) = s(x_2),$ we have $t(x_1) = t(x_2).$ Then there exists a rational map $h \colon Y \dashrightarrow Z$ such that $h \circ s= t$ on a suitable open subset of $X$.
\end{lemma}
\begin{proof}
Let $W \subset X \times Y \times Z$ be the subset
$$
W = \{ (x,y, z) \colon s(x) = y, t(x) = z \}.
$$
Notice that $W$ is the graph of the morphism
$$
\varphi \colon X \to Y \times Z, \ \ \ x \mapsto (s(x), t(x)),
$$
so $W$ is an irreducible variety isomorphic to $X$. Now consider the projection $\mathrm{pr}_{Y \times Z} \colon X \times Y \times Z \to Y \times Z$, and let $U = \mathrm{pr}_{Y \times Z} (W)$ and $V = \bar{U}$, where the bar denotes the Zariski closure. Then $V$ is an irreducible variety. Moreover, $U$ is constructible by (\cite{H}, Theorem 4.4), so in particular contains a dense open subset $P$ of $V$, which is itself an irreducible variety. Let now $p \colon P \to Y$ be the projection to the first component. We claim that $p$ gives a birational isomorphism between $P$ and $Y$. From our assumptions, we see that $p$ is dominant, and since char $K = 0,$ $p$ is also separable. So, it follows from (\cite{H}, Theorem 4.6) that to show that $p$ is birational, we only need to verify that it is injective. Suppose that $u_1 = (y_1, z_1), u_2 = (y_2, z_2) \in P$ with $y_1 = y_2.$ By our construction, there exist $x_1, x_2 \in X$ such that $s(x_1) = y_1$, $t(x_1) = z_1$ and $s(x_2) = y_2, t(x_2) = z_2.$ Since $s(x_1) = s(x_2),$ we have $t(x_1) = t(x_2),$ so $u_1 = u_2$, as needed.

Since $p$ is birational, we can now take $h = \pi_Z \circ p^{-1}\colon Y \dashrightarrow Z,$ where $\pi_Z \colon Y \times Z \to Z$ is the projection.
\end{proof}

Now let $\rho \colon E_n (R) \to GL_m (K)$ be a representation as in Theorem \ref{NC:T-NonComm1} or \ref{NC:T-NonComm2} and denote by $A$ the algebraic ring associated to $\rho.$ Also let $Q$ be the set of all pairs $(i,j)$ with $1 \leq i, j \leq n$, $i \neq j.$ Then, as we already observed at the beginning of this section, $E_n (A^{\circ})$ is the connected algebraic group generated by the images $W_{q} = \mathrm{im} \ \varphi_{q}$ of the regular maps
$$
\varphi_{q} \colon A^{\circ} \to GL_n (A^{\circ}), \ \ \ a \mapsto e_{q} (a),
$$
for all $q \in Q.$
In particular, (\cite{Bo}, Proposition 2.2) implies that there exists a finite sequence $(\alpha(1), \dots, \alpha(v))$ in $Q$ such that
$$
E_n (A^{\circ}) = W^{\varepsilon_1}_{\alpha(1)} \cdots W^{\varepsilon_v}_{\alpha(v)},
$$
where each $\varepsilon_i = \pm 1.$ Let
$$
X = \prod_{i=1}^v (A^{\circ})_{\alpha(i)}
$$
be the product of $v$ copies of $A^{\circ}$ indexed by the $\alpha(i)$ and define a regular map $s \colon X \to E_n (A^{\circ})$ by
\begin{equation}\label{NC:E-Reg}
s(a_{\alpha(1)}, \dots, a_{\alpha(v)}) = \varphi_{\alpha(1)}(a_{\alpha(1)})^{\varepsilon_1} \cdots \varphi_{\alpha(v)} (a_{\alpha(v)})^{\varepsilon_v}.
\end{equation}
Also let
\begin{equation}\label{NC:E-Reg1}
 t\colon X \to H^{\circ} \ \ \ t(a_{\alpha(1)}, \dots, a_{\alpha(v)}) =  \psi_{\alpha(1)}(a_{\alpha(1)})^{\varepsilon_1} \cdots \psi_{\alpha(v)} (a_{\alpha(v)})^{\varepsilon_v},
\end{equation}
where the $\psi_{\alpha(i)}$ are the regular maps from Proposition \ref{NC:P-AR}. With this set-up, we can now prove

\begin{prop}\label{NC:P-R2}
The homomorphisms $\sigma \colon E_n(A^{\circ}) \to H^{\circ}$ and $\bar{\sigma} \colon E_n (A^{\circ}) \to \bar{H}$
constructed in Propositions \ref{NC:P-SL} and \ref{NC:P-FI}, respectively, are morphisms of algebraic groups.
\end{prop}
\begin{proof}
We will only consider $\sigma$ as the argument for $\bar{\sigma}$ is completely analogous. Set
$Y = E_n (A^{\circ})$ and $Z = H^{\circ}$, and let $s \colon X \to Y$ and $t \colon X \to Z$ be the regular maps defined in (\ref{NC:E-Reg}) and (\ref{NC:E-Reg1}).
From the construction of $\sigma,$ it is clear that
$(\sigma \circ s) (x) = t(x),$ so in particular
$s(x_1) = s(x_2)$ for $x_1, x_2 \in X$ implies that
$t(x_1) = t(x_2).$ Hence, by Lemma \ref{NC:L-R1}, $\sigma$ is a rational map. Therefore, there exists
an open subset $V \subset E_n (A^{\circ})$ such that $\sigma \vert_V$ is regular. So, it follows from the next lemma
that $\sigma \colon E_n (A^{\circ}) \to H^{\circ}$ is a morphism.
\end{proof}

\begin{lemma}\label{L:R-3} {\rm (\cite{IR}, Lemma 6.4)}
Let $K$ be an algebraically closed field and let $\mathscr{G}$ and $\mathscr{G}'$ be affine algebraic groups over $K$, with $\mathscr{G}$ connected. Suppose $f \colon \mathscr{G} \to \mathscr{G}'$ is an abstract group homomorphism\footnotemark and assume there exists a Zariski-open set $V \subset \mathscr{G}$
such that $\varphi := f \vert_V$ is a regular map. Then $f$ is a morphism of algebraic groups. \footnotetext{Here we tacitly identify $\mathscr{G}$ and $\mathscr{G}'$ with the corresponding groups $\mathscr{G}(K)$ and $\mathscr{G}'(K)$ of $K$-points.}
\end{lemma}

Theorem \ref{NC:T-NonComm1} now follows from
Propositions \ref{NC:P-SL} and \ref{NC:P-R2} with $\B = A^{\circ} (= A).$ For Theorem \ref{NC:T-NonComm2}, we again take $\B = A^{\circ}$, and it remains to show that one can lift the morphism
$\bar{\sigma} \colon E_n (A^{\circ}) \to \bar{H}$ to a morphism $\sigma \colon E_n (A^{\circ}) \to H^{\circ}$ making the diagram (\ref{NC:D-3}) commute.
This accomplished through a suitable modification of the argument used in the proof of (\cite{IR}, Proposition 6.6). For this, we need some analogs of results contained in (\cite{IR}, \S 6) regarding the structure of
$E_n (B)$ as an algebraic $K$-group, where $B$ is an arbitrary finite-dimensional algebra over an algebraically closed field $K$.
Let $J = J(B)$ be the Jacobson radical of $B$. Then by the Wedderburn-Malcev Theorem (cf. \cite{P}, Corollary 11.6), there exists a semisimple subalgebra $\bar{B} \subset B$ such that $B = \bar{B} \oplus J$ as $K$-vector spaces and $\bar{B} \simeq B/J$ as $K$-algebras. Furthermore, since $K$ is algebraically closed, the Artin-Wedderburn Theorem implies that
$$
\bar{B} = M_{n_1} (K) \times \cdots \times M_{n_r} (K).
$$
Now consider the group homomorphism
$E_n (B) \to E_n (\bar{B})$ induced by the canonical map $B \to B/J$ (notice that this is a morphism of algebraic groups as $B \to B/J$ is a homomorphism of algebraic rings --- cf. \cite{IR}, Lemma 2.9), and let $E_n (J)$ be its kernel.
It is clear that $E_n (J)$ is a closed normal subgroup of $E_n (B).$ Note that
$$
E_n(M_{n_i} (K)) \simeq E_{nn_i} (K) \simeq SL_{nn_i} (K),
$$
so $E_n (\bar{B})$ is a semisimple simply-connected algebraic group. It is also easy to see that for any $a, b \geq 1,$ we have
$$
[GL_n (B, J^a), GL_n (B, J^b)] \subset GL_n (B, J^{a+b}),
$$
where $GL_n (B, J^s) = \ker (GL_n (B) \to GL_n (B/J^s)).$ Since $J$ is a nilpotent ideal, it follows that $E_n (J)$ is a nilpotent group. In particular, we obtain that
\begin{equation}\label{E:Levi}
E_n (B) = E_n (J) \rtimes E(\bar{B})
\end{equation}
is a Levi decomposition of $E_n (B)$ (cf. \cite{IR}, Proposition 6.5).

Now, using the Levi decomposition (\ref{E:Levi}) for $B = \B$, as well as the fact that the center
$Z(H^{\circ})$ is finite (Proposition \ref{NC:P-St3}), one can directly imitate the argument of (\cite{IR}, Proposition 6.6) to conclude the proof of Theorem \ref{NC:T-NonComm2}.

Finally, to derive Theorem 1 from Theorem \ref{NC:T-NonComm1}, we first note that by Lemma \ref{KT:L-DA}, we have $K$-algebra isomorphisms
$$
\B \simeq D \otimes_k C \simeq M_s (C)
$$
where $s^2 = \dim_k D$ and $C = \overline{f(k)}$ (as above, $f \colon D \to \B$ is the ring homomorphism associated to $\rho$). Consequently, $E_n (\B) \simeq E_n (M_s (C)) \simeq E_{ns} (C).$ Moreover, since $C$ is a finite-dimensional $K$-algebra, in particular a semilocal commutative ring, $E_{ns} (C) \simeq SL_{ns} (C)$ (cf. \cite{Mat}, Corollary 2). So, using the fact that $G = {\bf SL}_{n,D}$ is $K$-isomorphic to $SL_{ns}$ (\cite{PR}, 2.3.1),  we see that $E_n (\B) \simeq G(C).$ Letting $f_C \colon k \to C$ be the restriction of $f$ to $k$, we now obtain Theorem 1.

\vskip5mm

\section{Applications to representation varieties and deformations of representations}\label{S:D}

In this section, we will prove Theorem 2. To estimate the dimension of the character variety $X_n (\Gamma)$ for an elementary subgroup $\Gamma$ as in the statement of Theorem 2, we will exploit the well-known connection, going back to A.~Weil, between the tangent space of $X_n (\Gamma)$ at a given point and the 1-cohomology of $\Gamma$ with coefficients in the space of a naturally associated representation. We then use the results of \cite{IR} on standard descriptions of representations of $\Gamma$ to relate the latter space to a certain space of derivations of the finitely generated ring $R$ used to define $\Gamma$ (cf. Proposition \ref{D:P-1}). Since the dimensions of spaces of derivations are finite and are bounded by a constant depending only on $R$, we obtain the required bound on $\dim X_n (\Gamma)$.
Throughout this section, we will work over a fixed algebraically closed field $K$ of characteristic 0.

We begin by summarizing some key definitions and basic properties related to representation and character varieties, mostly following the first two chapters of the paper of Lubotzky-Magid \cite{LM}.
Let $\Gamma$ be a finitely generated group and fix an integer $n \geq 1.$ Recall that the $n^{\mathrm{th}}$ {\it representation scheme} of $\Gamma$ is the functor $\fR_n (\Gamma)$ from the category of commutative $K$-algebras to the category of sets defined by
$$
\fR_n (\Gamma) (A) = \Hom (\Gamma, GL_n (A)).
$$
More generally, if $\mathcal{G}$ is a linear algebraic group over $K$, we let the representation scheme of $\Gamma$ with values in $\G$ be the functor $\fR (\Gamma, \G)$ defined by
$$
\fR (\Gamma, \G)(A) = \Hom (\Gamma, \G(A)).
$$
Using the fact that for any commutative $K$-algebra $A$, a homomorphism $\rho \colon \Gamma \to GL_n (A)$ is determined by the images of the generators, subject to the defining relations of $\Gamma$, one shows
that $\fR_n (\Gamma)$ is an affine $K$-scheme represented by a finitely-generated $K$-algebra $\mathfrak{A}_n (\Gamma).$ Similarly, $\fR (\Gamma, \G)$ is an affine $K$-scheme represented by a finitely-generated $K$-algebra $\mathfrak{A} (\Gamma, \G)$ (cf. \cite{LM}, Proposition 1.2). The set $\fR_n (\Gamma)(K)$ of $K$-points of $\fR_n (\Gamma)$ is then denoted by $R_n (\Gamma)$, and is called the $n^{\mathrm{th}}$~{\it representation variety} of $\Gamma.$ It is an affine variety over $K$ with coordinate ring $A_n (\Gamma) = \mathfrak{A}_n (\Gamma)_{\mathrm{red}},$ the quotient of $\mathfrak{A}_n (\Gamma)$ by its nilradical. The representation variety $R(\Gamma, \G)$ is defined analogously.

Now let $\rho_0 \in R(\Gamma, \G).$ To describe the Zariski tangent space of
$\fR (\Gamma, \G)$ at $\rho_0$, denoted by $T_{\rho_0} (\fR (\Gamma, \G))$, we will use the algebra of dual numbers $K[\varepsilon]$ (where $\varepsilon^2 = 0$). More specifically, it is well-known that $\fR (\Gamma, \G) (K[\varepsilon])$ is the tangent bundle of $\fR (\Gamma, \G)$, and therefore $T_{\rho_0} (\fR (\Gamma, \G))$
can be identified with the fibre over $\rho_0$ of the map $\mu \colon \fR (\Gamma, \G) (K[\varepsilon]) \to \fR (\Gamma, \G) (K)$ induced by the augmentation homomorphism $K[\varepsilon] \to K, \varepsilon \mapsto 0$ (cf. \cite{Bo}, AG 16.2).
In other words, we have
$$
T_{\rho_0}(\fR (\Gamma, \G)) =  \{ \rho \in \mathrm{Hom} (\Gamma, \G(K[\varepsilon])) \mid \mu \circ \rho = \rho_0 \}.
$$
For us, it will be useful to have the following alternative description of $T_{\rho_0} (\fR (\Gamma, \G)).$
Let $\tilde{\g}$ be the Lie algebra of $\G.$ Notice that $\tilde{\g}$ has a natural $\Gamma$-action given by
$$
\gamma \cdot x = \mathrm{Ad}(\rho_0 (\gamma)) x,
$$
for $\gamma \in \Gamma$ and $x \in \tilde{\g}$, where $\mathrm{Ad} \colon \G(K) \to GL(\tilde{\g})$ is the adjoint representation.
Now
$T_{\rho_0} (\fR (\Gamma, \G))$ can be identified with the space $Z^1 (\Gamma, \tilde{\g})$ of 1-cocycles (cf. \cite{LM}, Proposition 2.2). Indeed, an element $c \in Z^1 (\Gamma, \tilde{\g})$ is by definition a map $c \colon \Gamma \to \tilde{\g}$ such that
$$
c (\gamma_1 \gamma_2) = c(\gamma_1) + \mathrm{Ad}(\rho_0 (\gamma_1)) c(\gamma_2).
$$
On the other hand, we have an isomorphism
$\G (K[\varepsilon]) \simeq \tilde{\g} \rtimes \G$ given by
$$
B + C \varepsilon \mapsto (CB^{-1}, B).
$$
Hence an element $\rho \in  T_{\rho_0}(\fR(\Gamma, \G))$ is a homomorphism $\rho \colon \Gamma \to \tilde{\g} \rtimes \G$ whose projection to the second factor is $\rho_0.$ In other words, it arises from a map
$c \colon \Gamma \to \tilde{\g}$ such that the map
$$
\Gamma \to \tilde{\g} \rtimes \G, \ \ \ \  \gamma \mapsto (c (\gamma), \rho_0(\gamma))
$$
is a group homomorphism. With the above identification, this translates into the condition
$$
c (\gamma_1 \gamma_2) = c(\gamma_1) + \mathrm{Ad}(\rho_0 (\gamma_1)) c(\gamma_2),
$$
giving the required isomorphism of $T_{\rho_0} (\fR (\Gamma, \G))$ with $Z^1 (\Gamma, \tilde{\g}).$
Also notice that for any finite-index subgroup $\Delta \subset \Gamma$ (which is automatically finitely generated), the natural restriction maps $\fR (\Gamma, \G) \to \fR (\Delta, \G)$ and $Z^1 (\Gamma, \tilde{\g}) \to Z^1 (\Delta, \tilde{\g})$ induce a commutative diagram
$$
\xymatrix{T_{\rho_0} (\fR (\Gamma, \G)) \ar[r] \ar[d] & Z^1 (\Gamma, \tilde{\g}) \ar[d] \\ T_{\rho_0} (\fR (\Delta, \G)) \ar[r] & Z^1 (\Delta, \tilde{\g})}
$$
where the horizontal maps are the isomorphisms described above.

Next, let us recall a characterization of the space $B^1 (\Gamma, \tilde{\g})$ of 1-coboundaries that will be used later; for this, we need to consider the action of $\G(K)$ on $R (\Gamma, \G).$ Given $\rho_0 \in R (\Gamma, \G)$,
let $\psi_{\rho_0} \colon \G(K) \to R (\Gamma, \G)$ be the orbit map, i.e. the map defined by
$$
\psi_{\rho_0}(T) = T \rho_0 T^{-1}, \ \ \  T \in \G (K).
$$
By direct computation, one shows that under the isomorphism
$T_{\rho_0} (\fR (\Gamma, \G)) \simeq Z^1 (\Gamma, \tilde{\g})$, the image of the differential
$(d \psi_{\rho_0})_e \colon T_{e} (\G) \to T_{\rho_0} (R (\Gamma, \G)) \subset T_{\rho_0} (\fR (\Gamma, \G))$
consists of maps $\tau \colon \Gamma \to \tilde{\g}$ such that there exists $A \in \tilde{\g}$ with
$$
\tau (\gamma) = A - \mathrm{Ad}(\rho_0 (\gamma)) A
$$
for all $\gamma \in \Gamma$, i.e. the image coincides with
$B^1 (\Gamma, \tilde{\g})$ (cf. \cite{LM}, Proposition 2.3). In fact, if $O (\rho_0)$ is the orbit of $\rho_0$ in $R(\Gamma, \G)$ under the action of $\G(K)$, then $B^1 (\Gamma, \tilde{\g})$ can be identified with $T_{\rho_0} (O (\rho_0)) \subset T_{\rho_0} (R(\Gamma, \G))$ (\cite{LM}, Corollary 2.4).

As a special case of the above constructions, we can consider
the action of $GL_n (K)$ on $R_n (\Gamma).$ The $n^{\mathrm{th}}$ {\it character variety} of $\Gamma$, denoted $X_n (\Gamma)$, is by definition the (categorical) quotient of $R_n (\Gamma)$ by $GL_n (K)$, i.e. it is the affine $K$-variety with coordinate ring $A_n (\Gamma)^{GL_n (K)}.$ Let $\pi \colon R_n (\Gamma) \to X_n (\Gamma)$ be the canonical map. Then each fiber $\pi^{-1} (x)$ contains a semisimple representation, and moreover if $\rho_1, \rho_2 \in R_n (\Gamma)$ are semisimple with $\pi (\rho_1) = \pi (\rho_2),$ then $\rho_1 = T \rho_2 T^{-1}$ for some $T \in GL_n (K).$ In particular, we see that $\pi$ induces a bijection between the isomorphism classes of semisimple representations and the points of $X_n (\Gamma)$ (cf. \cite{LM}, Theorem 1.28).

Now let us turn to the proof of Theorem 2. In the remainder of this section, $\Gamma$ will be the elementary subgroup $E(\Phi, R) \subset G(R)$, where $\Phi$
is a reduced irreducible root system of rank $\geq 2$, $G$ a universal Chevalley-Demazure group scheme of type $\Phi$, and $R$ a finitely generated commutative ring such that $(\Phi, R)$ is a nice pair.
By recent work of Ershov, Jaikin, and Kassabov \cite{EJK}, it is known that $\Gamma$ has Kazhdan's property (T). In particular, $\Gamma$ is finitely generated and satisfies the condition

\vskip2mm

\noindent (FAb) \ \ \parbox[t]{15cm}{for any finite-index subgroup $\Delta \subset \Gamma,$ the abelianization $\Delta^{\mathrm{ab}} = \Delta / [\Delta, \Delta]$ is finite}

\vskip2mm
\noindent (cf. \cite{HV}). This has the following consequence.
\begin{prop}\label{D:P-KG}
(cf. \cite{AR}, Proposition 2) Let $\Gamma$ be a group satisfying {\rm (FAb)}. For any $n \geq 1$, there exists a finite collection $G_1, \dots, G_d$ of algebraic subgroups of $GL_n (K)$, such that for any completely reducible representation $\rho \colon \Gamma \to GL_n (K)$, the Zariski closure $\overline{\rho (\Gamma)}$ is conjugate to one of the $G_i.$ Moreover, for each $i$, the connected component $G_i^{\circ}$ is a semisimple group.
\end{prop}
Thus, if $R_n (\Gamma)_{ss}$ denotes the set of completely reducible representations $\rho \colon \Gamma \to GL_n (K)$, then we have\footnotemark \footnotetext{Observe that if $\G \subset GL_n (K)$ is an algebraic subgroup such that $\G^{\circ}$ is semisimple, then $\G$ is completely reducible, hence any representation $\rho \colon \Gamma \to GL_n (K)$ with $\overline{\rho(\Gamma)} = \G$ is completely reducible.}
$$
R_n (\Gamma)_{ss} = \bigcup_{\substack{ i \in \{1, \dots, d \}, \\ g \in GL_n(K)}} g R'(\Gamma, G_i) g^{-1},
$$
where for an algebraic subgroup $\mathcal{G} \subset GL_n(K),$ we set
$$
R' (\Gamma, \G) = \{ \rho \colon \Gamma \to \G \mid \overline{\rho(\Gamma)} = \G \}.
$$
Therefore, letting $\pi \colon R_n (\Gamma) \to X_n (\Gamma)$ be the canonical map, we obtain that
\begin{equation}\label{E:D-SS1}
X_n (\Gamma) = \bigcup_{i=1}^d \pi (R' (\Gamma, G_i)).
\end{equation}
Notice that if $\G \subset GL_n (K)$ is an algebraic group such that $\G^{\circ}$ is semisimple, then $R' (\Gamma, \G)$ is an open subvariety of $R (\Gamma, \G)$. Indeed, let
$$
R'' (\Gamma, \G) = \{ \rho \colon \Gamma \to \G \mid \overline{\rho (\Gamma)} \supset \G^{\circ} \}.
$$
Since $\G^{\circ}$ is semisimple, $R'' (\Gamma, \G)$ is easily seen to be an open in $R(\Gamma, \G)$ (cf. \cite{AR1}, Lemma 4). On the other hand, we obviously have
$$
R' (\Gamma, \G) = R'' (\Gamma, \G) \cap (R (\Gamma, \G) \setminus \cup_{i=1}^{\ell} R(\Gamma, \mathcal{H}_i)),
$$
where $\mathcal{H}_1, \dots, \mathcal{H}_{\ell}$ are the algebraic subgroups of $\G$ such that
$$
\G \supsetneq \mathcal{H}_i \supset \G^{\circ}.
$$
Now let $W \subset X_n (\Gamma)$ be an irreducible component of maximal dimension, so that $\dim X_n (\Gamma) = \dim W.$ Then it follows from (\ref{E:D-SS1}) that we can find an irreducible component $V$ of some $R' (\Gamma, G_i)$ such that $\overline{\pi (V)} = W.$ Since $\pi \vert_{V}$ is dominant and separable (as char $K = 0$), it follows from (\cite{Bo}, AG 17.3) that there exists $\rho_0 \in V$ which is a simple point (of $R'(\Gamma, G_i)$) such that $\pi (\rho_0)$ is simple and the differential
\begin{equation}\label{E:D-Diff1}
(d \pi)_{\rho_0} \colon T_{\rho_0} (V) \to T_{\pi (\rho_0)} (W)
\end{equation}
is surjective. Next, let $\psi_{\rho_0} \colon G_i \to R(\Gamma, G_i)$ be the orbit map. By the construction of $\pi$, we have $(\pi \circ \psi_{\rho_0} )(T) = \pi (\rho_0)$ for any $T \in G_i$, so $d(\pi \circ \psi_{\rho_0})_e = 0.$
On the other hand, as we noted above,
the image of the differential $(d \psi_{\rho_0})_e$ is the space $B = B^1 (\Gamma, \tilde{\g}_i)$, where $\tilde{\g}_i$ is the Lie algebra of $G_i$ with $\Gamma$-action given by $\mathrm{Ad} \circ \rho_0$. Since $\rho_0$ is a simple point, it lies on a unique irreducible component of $R'(\Gamma, G_i)$, so it follows that the image of $\psi_{\rho_0}$ (i.e. the orbit of $\rho_0$) is contained in $V$. Consequently, (\ref{E:D-Diff1}) factors through
$$
T_{\rho_0} (V)/ B \to T_{\pi (\rho_0)} (W).
$$
Since obviously $\dim_K T_{\rho_0} (V) \leq \dim_K T_{\rho_0} (\fR(\Gamma, G_i))$ and $T_{\rho_0} (\fR (\Gamma, G_i)) \simeq Z^1 (\Gamma, \tilde{\g})$,
we therefore obtain that
\begin{equation}\label{E:D-TSpDim}
\dim X_n (\Gamma) = \dim W \leq \dim_K H^1 (\Gamma, \tilde{\g}_i).
\end{equation}
Thus, the proof of Theorem 2 is now reduced to considering the following situation.
Let $\rho_0 \colon \Gamma \to GL_n (K)$ be a representation, set $\G = \overline{\rho_0 (\Gamma)}$, and let $\tilde{\g}$ be the Lie algebra of $\G$, considered as a $\Gamma$-module via $\mathrm{Ad} \circ \rho_0.$
Assume that the connected component $\G^{\circ}$ is semisimple. Then we need to give an upper bound on $\dim_K H^1 (\Gamma, \tilde{\g}).$ This will be made more precise in Proposition \ref{D:P-1} below, after some preparatory remarks.

First, notice that for the purpose of estimating $\dim_K H^1 (\Gamma, \tilde{\g}),$ we may compose $\rho_0$ with the adjoint representation and assume without loss of generality that the group $\G$ is adjoint.
Now, since $\G^{\circ}$ is semisimple, $\rho_0$ has a standard description by (\cite{IR}, Theorem 6.7),
i.e.
there exists a commutative finite-dimensional $K$-algebra $A_0$,
a ring homomorphism
\begin{equation}\label{E:2}
f_0 \colon R \to A_0
\end{equation}
with Zariski-dense image, and a morphism of algebraic groups
\begin{equation}\label{E:3}
\theta \colon G(A_0) \to \G
\end{equation}
such that on a suitable finite-index subgroup $\Delta \subset \Gamma$, we have
\begin{equation}\label{E:1}
\rho_0 \vert_{\Delta} = (\theta \circ F_0) \vert_{\Delta},
\end{equation}
where $F_0 \colon \Gamma \to G(A_0)$ is the group homomorphism induced by $f_0$. Moreover, it follows from (\cite{IR}, Proposition 5.3) that $\theta (G(A_0)) = \G^{\circ}$.




Next, let $\G_1, \dots, \G_r$ be the (almost) simple components of $\G^{\circ}$ (cf. \cite{Bo}, Proposition 14.10).
Since $\G^{\circ}$ is adjoint,
the product map
$$
\G_1 \times \cdots \times \G_r \to \G^{\circ}
$$
is an isomorphism. The following lemma gives a more detailed description of $A_0.$

\begin{lemma}\label{D:L-1}
The algebraic ring $A_0$ is isomorphic to the product
$
\underbrace{K \times \cdots \times K}_{r \ \mathrm{copies}}.
$
\end{lemma}
\begin{proof}
Let $J_0$ be the Jacobson radical of $A_0.$ Since $\G^{\circ}$ is semisimple (in particular, reductive), $J_0 = \{ 0 \}$ by (\cite{IR}, Lemma 5.7), and consequently by (\cite{IR}, Proposition 2.20), we have
$$
A_0 \simeq K^{(1)} \times \cdots \times K^{(s)},
$$
where $K^{(i)} \simeq K$ for all $i.$
Thus $G(A_0) = G(K^{(1)}) \times \cdots \times G(K^{(s)}).$ As we observed above, the map $\theta$ is surjective, so, since $G(K)$ is an almost simple group, it follows that $s \geq r.$ On the other hand, by (\cite{IR}, Theorem 3.1), for each root $\alpha \in \Phi$, there exists an injective map $\psi_{\alpha} \colon A_0 \to \G$ such that
\begin{equation}\label{D:E-Inj1}
\theta (e_{\alpha} (a)) = \psi_{\alpha} (a),
\end{equation}
where $e_{\alpha} (A_0)$ is the 1-parameter root subgroup of $G(A_0)$ corresponding to the root $\alpha$ (cf. \cite{IR}, Proposition 4.2).
Now if $s > r$, then $\theta$ would kill some simple component $G(K^{(i)})$ of $G(A_0).$ Since $G(K^{(i)})$ intersects each root subgroup $e_{\alpha} (A_0),$ the maps $\psi_{\alpha}$ would not be injective, a contradiction. So, $s = r$, as claimed.
\end{proof}

Thus, we can write $f_0 \colon R \to A_0$ as
\begin{equation}\label{D:E-hom}
f_0 (t) = (f_0^{(1)} (t), \dots, f_0^{(r)}(t)),
\end{equation}
for some ring homomorphisms $f_0^{(i)} \colon R \to K$.

\vskip2mm

\noindent {\bf Remark 4.3.} Notice that for each $i$, the image $\theta(G(K^{(i)}))$ intersects a unique simple factor of $\G^{\circ}$, say $\theta(G(K^{(i)})) \cap \G_i \neq \{ e \}$, and then $\theta(G(K^{(i)})) = \G_i.$ Furthermore, it follows from the proof of Lemma \ref{D:L-1} that $\theta$ is an isogeny,
so since char $K = 0$, the differential $(d \theta)_{e} \colon \g \to \tilde{\g}_i$ gives an isomorphism of Lie algebras. In particular, we see that the Lie algebras of all the simple factors $\G_i$ are isomorphic (in fact, they are isomorphic as $G(K)$-modules, with $G(K)$ acting via $\mathrm{Ad} \circ \theta$).


\vskip2mm

To formulate the next result, we need to introduce the following notation. Suppose $g \colon R \to K$ is a ring homomorphism. Then we will denote by $\mathrm{Der}^g (R,K)$ the space of $K$-valued derivations of $R$ with respect to $g$, i.e. an element $\delta \in \mathrm{Der}^g (R,K)$
is a map $\delta \colon R \to K$ such that for any $r_1, r_2, \in R$,
$$
\delta (r_1 + r_2) = \delta (r_1) + \delta(r_2) \ \ \ \mathrm{and} \ \ \  \delta (r_1 r_2) = \delta(r_1) g(r_2) + g(r_1) \delta (r_2).
$$


\addtocounter{thm}{1}

We now have

\begin{prop}\label{D:P-1}
Suppose $\rho_0 \colon \Gamma \to GL_n (K)$ is a linear representation and set $\G = \overline{\rho_0 (\Gamma)}$. Denote by $\tilde{\g}$ the Lie algebra of $\G$ and assume that $\G^{\circ}$ is semisimple. Then
$$
\dim_K H^1 (\Gamma, \tilde{\g}) \leq \sum_{i=1}^r \dim_K \mathrm{Der}^{f_0^{(i)}} (R,K),
$$
where the $f_0^{(i)}$ are the ring homomorphisms appearing in $\mathrm{(\ref{D:E-hom})}$.
\end{prop}

We first note two facts that will be needed in the proof. Let $\Lambda \subset \Gamma$ be any finite-index subgroup. Then,
as we have already seen, the space of 1-cocycles $Z^1 (\Lambda, \tilde{\g})$ can be naturally identified with the tangent space
\begin{equation}\label{E:D-TSp}
T_{\rho_0}(\fR (\Lambda, \G)) =  \{ \rho \in \mathrm{Hom} (\Lambda, \G(K[\varepsilon])) \mid \mu \circ \rho = \rho_0 \}.
\end{equation}
Also observe that the restriction map
$$
\mathrm{res}_{\Gamma/ \Lambda} \colon H^1 (\Gamma, \tilde{\g}) \to H^1 (\Lambda, \tilde{\g})
$$
is injective. Indeed, since $[\Gamma : \Lambda] < \infty,$ the corestriction map $\mathrm{cor}_{\Gamma/ \Lambda} \colon H^1 (\Lambda, \tilde{\g}) \to H^1 (\Gamma, \tilde{\g})$ is defined and the composition $\mathrm{cor}_{\Gamma/ \Lambda} \circ \mathrm{res}_{\Gamma/ \Lambda}$ coincides with multiplication by $[\Gamma : \Lambda].$ Since char~$K = 0,$ the injectivity of $\mathrm{res}_{\Gamma/ \Lambda}$ follows.

\vskip2mm

\noindent {\it Proof of Proposition \ref{D:P-1}.} Set
$$
X = \mathrm{Der}^{f_0^{(1)}} (R, K) \oplus \cdots \oplus \mathrm{Der}^{f_0^{(r)}} (R, K).
$$
and let $\Delta \subset \Gamma$ be the finite-index subgroup appearing in (\ref{E:1}).
We will show that there exists a linear map
$\psi \colon  X \to H^1 (\Delta, \tilde{\g})$ such that
$
\mathrm{res}(H^1 (\Gamma, \tilde{\g})) \subset \mathrm{im} (\psi).
$
The proposition then follows from the injectivity of the restriction map.

The map $\psi$ is constructed as follows. Choose derivations $\delta_{i} \in \mathrm{Der}^{f_0^{(i)}} (R,K)$, for $i = 1, \dots, r,$ and let
$$
B = \underbrace{K[\varepsilon] \times \cdots \times K[\varepsilon]}_{r \ \mathrm{copies}}
$$
(with $\varepsilon^2 = 0$).
Then
$$
f_{\delta_1, \dots, \delta_r} \colon R \to B, \ \ \ s \mapsto (f_0^{(1)} (s) + \delta_1 (s) \varepsilon, \dots, f_0^{(r)} (s) + \delta_r (s) \varepsilon)
$$
is a ring homomorphism, hence induces a group homomorphism
$$
F_{\delta_1, \dots, \delta_r} \colon \Gamma \to G(B)
$$
(recall that $\Gamma = E(R) \subset G(R)$).
On the other hand, we have
$$
G(B) \simeq (\g \oplus \cdots \oplus \g) \rtimes (G(K) \times \cdots \times G(K)) \simeq \mathrm{Lie} (G(A_0)) \rtimes G(A_0)
$$
and
$$
\G (K [\varepsilon]) \simeq \tilde{\g} \rtimes \G,
$$
so we can define a group homomorphism $\tilde{\theta} \colon G(B) \to \G (K [\varepsilon])$ by the formula
$$
(x, g) \mapsto ((d \theta)_e (x), \theta (g)),
$$
where $\theta \colon G(A_0) \to \G$ is the morphism appearing in (\ref{E:3}). Notice that since by Remark 4.3, the differential of $\theta$ gives a homomorphism
$(d \theta)_e \colon \g \to \tilde{\g}_i$ for each factor $\g$ of $\mathrm{Lie}(G(A_0))$, the map
$\tilde{\theta}$ can also be described as follows: let $x_1, \dots, x_r \in \g$ and $g \in G(A_0).$ Then
$$
\tilde{\theta} (x_1, \dots, x_r, g) = \left( \sum_{i=1}^r (d \theta)_e (x_i), \theta (g) \right),
$$
Now, $\tilde{\theta} \circ F_{\delta_1, \dots, \delta_r}$ is a homomorphism $\Gamma \to \G (K[\varepsilon])$, and in view of (\ref{E:1}), we have $$\mu \circ (\tilde{\theta} \circ F_{\delta_1, \dots, \delta_r} \vert_{\Delta}) = \rho_0.$$
It follows from (\ref{E:D-TSp}) that
$$
c_{\delta_1, \dots, \delta_r} := \tilde{\theta} \circ \mathrm{pr} \circ F_{\delta_1, \dots, \delta_r} \vert_{\Delta},
$$
where $\mathrm{pr} \colon G(B) \to \mathrm{Lie}(G(A_0))$ is the projection, is an element of $Z^1 (\Delta, \tilde{\g}).$ Now put
$$
\psi((\delta_1, \dots, \delta_r)) = [c_{\delta_1, \dots, \delta_r}],
$$
where $[c_{\delta_1, \dots, \delta_r}]$ denotes the class of $c_{\delta_1, \dots, \delta_r}$ in $H^1 (\Delta, \tilde{\g}).$

Conversely, suppose $\rho \colon \Gamma \to \G (K [\varepsilon])$ is a homomorphism with $\mu \circ \rho = \rho_0$. By (\cite{IR}, Proposition 2.14 and Theorem 3.1), we can associative to $\rho$ a commutative finite-dimensional $K$-algebra $A$ together with a ring homomorphism $f \colon R \to A$ with Zariski-dense image.

\begin{lemma}\label{D:L-2}
Let $A$ be the finite-dimensional commutative $K$-algebra associated to $\rho.$ Then
$$
A \simeq \tilde{K}^{(1)} \times \cdots \times \tilde{K}^{(r)},
$$
where, as above, $r$ is the number of simple components of $\G^{\circ}$, and for each
$i$, $\tilde{K}^{(i)}$ is isomorphic to either $K$ or $K[\varepsilon]$ $($with $\varepsilon^2 = 0).$
\end{lemma}
\begin{proof}
Let $J$ be the Jacobson radical of $A.$ Observe that since the unipotent radical $U$ of $\overline{\rho (\Gamma)}^{\circ}$ is commutative (which follows from the fact that $\tilde{\g}$ is the unipotent radical of $\G (K[\varepsilon])$), we have $J^2 = \{ 0 \}$ by (\cite{IR}, Lemma 5.7). Now by our assumption, $\mu \circ \rho = \rho_0,$ where, $\mu \colon \G (K[\varepsilon]) \to \G (K)$ is the homomorphism induced by ring homomorphism $K[\varepsilon] \to K,$ $\varepsilon \mapsto 0.$ In particular, for any root $\alpha \in \Phi$, we have
\begin{equation}\label{D:E-RHom}
\mu (\rho (e_{\alpha} (r))) = \rho_0 (e_{\alpha} (r)),
\end{equation}
for all $r \in R.$ Since $\mu$ is a morphism of algebraic groups and the algebraic rings $A$ and $A_0$ associated to $\rho$ and $\rho_0,$ respectively, are by construction the connected components of $\overline{\rho(e_{\alpha} (R))}$ and $\overline{\rho_0 (e_{\alpha} (R))}$ for any root $\alpha$ (cf. \cite{IR}, Theorem 3.1), it follows that $\mu$ induces a surjective map $\nu \colon A \to A_0.$ Moreover, since (\ref{D:E-RHom}) holds for all roots $\alpha \in \Phi$, the construction of the ring operations on $A$ and $A_0$ given in (\cite{IR}, Theorem 3.1) implies that $\nu$ is actually a ring homomorphism. Also notice that since $J$ is commutative and nilpotent, we have
$J \subset \ker \nu$ by the definition of $\mu$. On the other hand, the ring $A_0$ is semisimple by Lemma \ref{D:L-1}, so $J = \ker \nu.$ Thus $A_0 \simeq A/J \simeq K \times \cdots \times K.$

Next, by the Wedderburn-Malcev Theorem, we can find a semisimple subalgebra $\tilde{B} \subset A$ such that $A = \tilde{B} \oplus J$ as $K$-vector spaces and $\tilde{B} \simeq A/J \simeq K \times \cdots \times K$ as $K$-algebras (cf. \cite{P}, Corollary 11.6). Let $e_i \in \tilde{B}$ be the $i$th standard basis vector. Since $e_1, \dots, e_r$ are idempotent, and we have $e_1 + \cdots + e_r = 1$ and $e_i e_j = 0$ for $i \neq j,$ it follows that we can write
$A = \oplus_{i=1}^{r} A_i,$ where $A_i = e_i A.$ Clearly, $A_i = \tilde{B}_i \oplus J_i$ with $\tilde{B}_i = e_i \tilde{B} \simeq K$ and $J_i = e_i J$; in particular, $A_i$ is a local $K$-algebra with maximal ideal $J_i$ such that $J_i^2 = \{ 0 \}.$ To complete the proof, it obviously suffices to show that $s_i := \dim_K J_i \leq 1$ for all $i.$

Now, by (\cite{IR}, Proposition 6.5), for each $i = 1, \dots, r$, we have a Levi decomposition
$$
G(A_i) = \underbrace{(\g \oplus \cdots \oplus \g)}_{s_i \ \mathrm{copies}} \rtimes G(K),
$$
where $\g$ is the Lie algebra of $G(K).$ Also, by (\cite{IR}, Theorem 6.7), there exists a morphism
\begin{equation}\label{E:D-10}
\sigma \colon G(A) \to \G(K[\varepsilon])
\end{equation}
of algebraic groups such that for a suitable subgroup of finite index $\Delta' \subset \Gamma$, we have
\begin{equation}\label{E:4}
\rho \vert_{\Delta'} = \sigma \circ F \vert_{\Delta'},
\end{equation}
where $F \colon \Gamma \to G(A)$ denotes the group homomorphism induced by $f.$ Since $\mu \circ \rho = \rho_0$, and for $\tilde{\Delta} = \Delta \cap \Delta',$ we have
$$
\rho_0 \vert_{\tilde{\Delta}} = (\theta \circ F_0) \vert_{\tilde{\Delta}} \ \ \ \mathrm{and} \ \ \ \rho \vert_{\tilde{\Delta}} = \sigma \circ F \vert_{\tilde{\Delta}}
$$
by (\ref{E:1}) and (\ref{E:4}), it follows that the diagram
\begin{equation}\label{D:E-CommDiag1}
\xymatrix{G(A) \ar[r]^{\sigma} \ar[d]^{\tilde{\nu}} & \G (K[\varepsilon]) \ar[d]_{\mu} \\ G(A_0) \ar[r]^{\theta} & \G}
\end{equation}
commutes (where $\tilde{\nu}$ is the homomorphism induced by $\nu$). Now Remark 4.3, together with the definition of $\nu$, implies that $(\theta \circ \tilde{\nu}) (G(A_i)) = \G_i$, where $\G_i$ is a simple factor of $\G.$ Since $G(A_i)$ coincides with its commutator subgroup (cf. \cite{St1}, Corollary 4.4), we obtain that $\sigma (G(A_i))$ is a subgroup of $\G (K[\varepsilon])$ that maps to $\G_i$ under $\mu$ and coincides with its commutator, so the fact that the simple factors $\G_1, \dots, \G_r$ of $\G$ commute elementwise implies that $\sigma (G(A_i)) \subset \tilde{\g}_i \rtimes \G_i,$ where $\g_i$ is the Lie algebra of $\G_i.$ On the other hand, by (\cite{IR}, Theorem 3.1),
for each root $\alpha \in \Phi$, there exists an injective map $\tilde{\psi_{\alpha}} \colon A \to \G (K[\varepsilon])$ such that
\begin{equation}
\sigma (e_{\alpha} (a)) = \tilde{\psi}_{\alpha} (a),
\end{equation}
where $e_{\alpha} (A)$ is the 1-parameter root subgroup of $G(A)$ corresponding to the root $\alpha$.
So, since $\tilde{\g}_i \simeq \g$ by Remark 4.3, the same argument as in the proof of Lemma \ref{D:L-1} shows that $s_i \leq 1.$


\end{proof}
For ease of notation, we will view $A$ as a subalgebra of
\begin{equation}\label{D:E-AR11}
\tilde{A} := \underbrace{K[\varepsilon] \times \cdots \times K[\varepsilon]}_{r \ \mathrm{copies}}
\end{equation}
Then, using the lemma and the assumption that $\mu \circ \rho = \rho_0$,
we can write the homomorphism $f \colon R \to A$ in the form
\begin{equation}\label{E:5}
f(t) = (f_0^{(1)} (t) + \delta_1 (t) \varepsilon, \dots, f_0^{(r)} (t) + \delta_{r} (t)\varepsilon)
\end{equation}
with $(\delta_1, \dots, \delta_r) \in X$ and $\delta_i = 0$ for $i = r_2+1, \dots, r.$

To describe the cohomology class corresponding to $\rho$, we will now need to analyze more closely the morphism
$\sigma$ introduced in (\ref{E:D-10}). First, we note that if $\bar{A} = A/J$ and $G(A,J)$ is the congruence subgroup
$$
G(A,J) = \ker (G(A) \to G(\bar{A})),
$$
then by (\cite{IR}, Proposition 6.5),
$$
G(A) = G(A,J) \rtimes G(\bar{A})
$$
is a Levi decomposition of $G(A)$. Now by (\cite{Bo}, Proposition 11.23), any two Levi subgroups of $(\G (K[\varepsilon]))^{\circ}$ are conjugate under an element of the unipotent radical $R_u (\G (K[\varepsilon]))^{\circ},$ which can be identified with $\G^{\circ} (K[\varepsilon], (\varepsilon)) \simeq \tilde{\g}.$ In our case, we can apply this to the groups $\sigma (G(\bar{A}))$ and $\theta (G(A_0)) = G^{\circ}$ (where $\theta$ is the morphism from (\ref{E:3})) to conclude that $B \theta(G(A_0)) B^{-1} = \sigma (G(\bar{A}))$ for some $B \in \G(K[\varepsilon], (\varepsilon)) \simeq \tilde{\g}.$ By direct computation, one sees that for any $X \in \G$ and $B = I + \varepsilon Y \in \G(K[\varepsilon], (\varepsilon))$,
$$
B X B^{-1} = (I + \varepsilon (Y - XYX^{-1}))X,
$$
which shows that
$$
\rho (\gamma) = \sigma( F(\gamma)) = ( (\sigma \circ \mathrm{pr} \circ F)(\gamma) + Y - \mathrm{Ad}(\theta (F_0 (\gamma))(Y), \theta (F_0 (\gamma)))
$$
for all $\gamma \in \tilde{\Delta} = \Delta \cap \Delta'$ (where $\Delta$ and $\Delta'$ are the finite-index subgroups of $\Gamma$ appearing in (\ref{E:1}) and (\ref{E:4}), respectively).
Since $\theta (F_0 (\gamma)) = \rho_0 (\gamma)$ for $\gamma \in \tilde{\Delta}$, we can rewrite this as
$$
\rho (\gamma) = (c(\gamma), \rho_0 (\gamma)),
$$
where
$$
c(\gamma) = (\sigma \circ \mathrm{pr} \circ F)(\gamma) + Y - \mathrm{Ad}(\rho_0(\gamma))(Y).
$$
Using (\ref{E:D-TSp}), we obtain
$c \in Z^1 (\tilde{\Delta}, \tilde{\g}).$ Now let $b_Y \in B^1 (\tilde{\Delta}, \tilde{\g})$ be the 1-coboundary defined by $b_Y (\gamma) = Y - \mathrm{Ad}(\rho_0(\gamma))Y$, and put $\tilde{c} = c - b_Y$ (thus $\tilde{c}$ and $c$ define the same element of $H^1(\tilde{\Delta}, \tilde{g})$). Then
$$
\tilde{c}(\gamma) = (\sigma \circ \mathrm{pr} \circ F)(\gamma)
$$
for all $\gamma \in \tilde{\Delta}.$
To complete the proof of the proposition, we will need the following lemma.

\begin{lemma}\label{D:L-SS}
Assume that $K$ is an algebraically closed field of characteristic 0.
Let $\pi \colon \G \to \G'$ be an isogeny of absolutely almost simple algebraic groups. Let $\g$ (resp. $\g'$) denote the Lie algebra of $\G$ (resp. $\G'$). Set
$$
\mathcal{H} = \g \rtimes \G \ \ \ and \ \ \ \mathcal{H}' = \g' \rtimes \G',
$$
where $\G$ (resp. $\G'$) acts on $\g$ (resp. $\g'$) via the adjoint representation. Then for any morphism $\varphi \colon \mathcal{H} \to \mathcal{H}'$ such that $\varphi \vert_{\G} = \pi$, there exists $a \in K$ such that
$$
\varphi(X, g) = (a (d\pi)_e (X), \pi (g)).
$$
\end{lemma}
\begin{proof}
Since char $K = 0$ and $\g$ and $\g'$ are simple Lie algebras, the adjoint representations $\mathrm{Ad} \colon \G \to GL (\g)$ and $\mathrm{Ad} \colon \G' \to GL (\g')$ are both irreducible. Let us now view $\g'$ as a $\G$-module, with $\G$ acting via $\pi.$
Then both $\varphi \vert_{\g}$ and $(d \pi)_e$ are $\G$-equivariant homomorphisms of irreducible $\G$-modules. So, by Schur's lemma, $\varphi \vert_{\g} = a (d \pi)_e$ for some $a \in K$ (cf. \cite{A}, Theorem 9.6).
\end{proof}

Now, as above, we consider $A$ as a subalgebra of the algebra $\tilde{A}$ appearing in (\ref{D:E-AR11}); after possible renumbering, we may assume that, in the notation of Lemma \ref{D:L-2}, we have $\tilde{K}^{(i)} \simeq K[\varepsilon]$ for $i = 1, \dots, s$, where $s = \dim_K J(A)$, and $\tilde{K}^{(i)} \simeq K$ for $i = s+1, \dots, r.$ We will view
$G(A)$ as a subgroup of
$$
G(\tilde{A}) \simeq \mathrm{Lie}(G(A_0)) \rtimes G(A_0),
$$
and write $G(A_0) = G(K^{(1)}) \times \cdots \times G(K^{(r)})$ and $\mathrm{Lie} (G(A_0)) = \g_1 \oplus \cdots \oplus g_r,$ where $G(K^{(i)}) = G(K)$ and $\g_i = \g$ for all $i.$ We will also regard $\sigma \colon G(A) \to G(K[\varepsilon])$ as a morphism
$\sigma \colon G(\tilde{A}) \to \G(K[\varepsilon])$, with $\sigma \vert_{\g_i} = 0$ for all $i > s.$
Now since by our construction, the cocycles $c$ and $\tilde{c}$ lie in the same cohomology class, we may assume without loss of generality that $\sigma$ has the form
$$
\sigma (x_1, \dots, x_r, g) = (\sigma \vert_{\g_1 \oplus \cdots \oplus \g_r} (x_1, \dots, x_r), \theta (g)),
$$
for $(x_1, \dots, x_r, g) \in (\g_1 \oplus \cdots \oplus \g_r) \rtimes G(\bar{A}).$
By Remark 4.3, for each factor $G(K^{(i)})$ of $G(A_0)$,
the differential $(d \theta)_e \colon \g_i \to \tilde{\g}_i$ yields an isomorphism of $G(K)$-modules (with $G(K)$ acting on $\tilde{\g}_i$ via $\mathrm{Ad} \circ \theta$). Furthermore, since $\sigma \vert_{G(\bar{A})} = \theta$,
the same argument as used in the proof of Lemma \ref{D:L-2} shows that
$\sigma (\g_i) = \tilde{\g}_i$, for $i = 1, \dots, s.$
Now applying Lemma \ref{D:L-SS} to the restrictions $\sigma \vert_{\g_i \rtimes G(K^{(i)})}$ and $((d \theta)_e, \theta) \vert_{\g_i \rtimes G(K^{(i)})}$, we obtain
$$
\sigma \vert_{\g_i} = a (d \theta)_e \vert_{\g_i}
$$
for some $a \in K$ (possibly 0). Repeating this for all factors shows that for $(x_1, \dots, x_r) \in \g_1 \oplus \cdots \oplus \g_r,$ we have
$$
\sigma \vert_{\g_1 \oplus \cdots \oplus \g_r} (x_1, \dots x_r) = \sum_{i=1}^r a_i (d \theta)_e (x_i).
$$
So, replacing the element $(\delta_1, \dots, \delta_r)$ in (\ref{E:5}) by $(a_1 \delta_1, \dots, a_r \delta_r)$,
we have
$$
\tilde{c}(\gamma) = c_{\delta_1, \dots, \delta_r} (\gamma)
$$
for all $\gamma \in \tilde{\Delta}.$ Now let $\psi((\delta_1, \dots, \delta_r)) = d_{\delta_1, \dots, \delta_r} \in Z^1 (\Delta, \tilde{\g})$ and let $c_{\rho}$ be the element of $Z^1(\Gamma, \tilde{\g})$ corresponding to $\rho.$ It follows that
$$
\mathrm{res}_{\Delta/ \tilde{\Delta}} (\mathrm{res}_{\Gamma / \Delta} ([c_{\rho}])) = \mathrm{res}_{\Delta / \tilde{\Delta}} ([d_{\delta_1, \dots, \delta_r}]),
$$
where
$$
\mathrm{res}_{\Gamma / \Delta} \colon H^1 (\Gamma, \tilde{\g}) \to H^1 (\Delta, \tilde{\g}) \ \ \ \mathrm{and} \ \ \ \mathrm{res}_{\Delta/ \tilde{\Delta}} \colon H^1 (\Delta, \tilde{\g}) \to H^1 (\tilde{\Delta}, \tilde{\g})
$$
are the restriction maps. So, the injectivity of the restriction maps yields
$$
\mathrm{res}_{\Gamma / \Delta} ([c_{\rho}]) = [d_{\delta_1, \dots, \delta_r}],
$$
which shows that
$$\mathrm{res}(H^1 (\Gamma, \tilde{\g})) \subset \mathrm{im} (\psi).$$ This completes the proof of the proposition.
$\hfill \Box$

\vskip3mm

\noindent {\it Proof of Theorem 2.}
In view of (\ref{E:D-TSpDim}) and Proposition \ref{D:P-1}, it remains to show that $r \leq n$ and to give a bound on the dimension of the space $\mathrm{Der}^g (R, K)$, for any ring homomorphism $g \colon R \to K,$ which is independent of $g$. Notice that
$\G^{\circ} \subset GL_n (K)$ and $\G^{\circ} = \G_1 \times \dots \times \G_r,$ we have
$$
n \geq \mathrm{rk} \G^{\circ} = \sum_{i=1}^r \mathrm{rk} \G_i \geq r,
$$
as needed. For the second task, we have the following (elementary) lemma.
\begin{lemma}\label{D:L-Der}
Let $R$ be a finitely generated commutative ring, and denote by $d$ the minimal number of generators of $R$ (i.e. the smallest integer such that there exists a surjection $\Z[x_1, \dots, x_d] \twoheadrightarrow R$). Then for any field $K$ and ring homomorphism $g \colon R \to K$, $\dim_K \mathrm{Der}^g (R, K) \leq d.$ If, moreover, $K$ is a field of characteristic 0, $R$ is an integral domain with field of fractions $L$, and
$g$ is injective, then $\dim_K \mathrm{Der}^g (R, K) \leq \ell,$ where $\ell$ is the transcendence degree of $L$ over its prime subfield.
\end{lemma}
\begin{proof}
Let $S= \{r_1, \dots, r_d\}$ be a minimal set of generators of $R$. Since any element $\delta \in \mathrm{Der}^g (R,K)$ is completely determined by its values on the elements of $S$, the map
$$
\delta \mapsto (\delta (r_1), \dots, \delta (r_d))
$$
defines an injection $\mathrm{Der}^g (R,K) \to K^d,$ so $\dim_K \mathrm{Der}^g (R, K) \leq d,$ as claimed.

Now suppose that $R$ is a finitely generated integral domain and $g$ is injective. Since char $K = 0$,
after possibly localizing $R$ with respect to the multiplicative set $\Z \setminus \{ 0 \}$ (which does not affect the dimension of
the space $\mathrm{Der}^g (R, K)$), we can use Noether's normalization lemma to write $R$
as an integral extension of $S = \Q [x_1, \dots, x_{\ell}]$ so that the field of fractions of $R$ is a separable extension of that of $S$.
Combining this with
the assumption that $g$ is injective, one easily sees that any derivation $\delta$ of $R$ is uniquely determined by its restriction to $S$ (cf. \cite{La}, Ch. VII, Theorem 5.1), so, in particular,
$$
\dim_K \mathrm{Der}^g (R,K) \leq \dim_K \mathrm{Der}^g (S,K) =: s.
$$
On the other hand, the argument given in the previous paragraph paragraph shows that $s \leq \ell,$ which completes the proof.
\end{proof}

\vskip2mm

\noindent {\bf Remark 4.8.} Notice that the estimate $\dim_K \mathrm{Der}^g (R,K) \leq \ell$ may not be true if $g$ is not injective. Indeed, take $K = \bar{\Q}$, and let $R_0 = \Z[X,Y]$ and $R = \Z[X,Y]/(X^3 - Y^2).$ Furthermore, let
$$
f \colon \Z[X,Y] \to \bar{\Q}, \ \ \varphi(X,Y) \mapsto \varphi (0,0)
$$
and denote by $g \colon R \to \bar{\Q}$ the induced homomorphism. The space $\mathrm{Der}^f (R_0, \bar{\Q})$ is spanned by the linearly-independent derivations $\delta_x$ and $\delta_y$ defined by
$$
\delta_x (\varphi(X,Y)) = \frac{\partial \varphi}{\partial X} (0,0) \ \ \ \delta_y (\varphi(X,Y)) = \frac{\partial \varphi}{\partial Y} (0,0),
$$
so $\dim_{\bar{\Q}} \mathrm{Der}^f (R_0, \bar{\Q}) = 2$. Now notice that the natural map
$$
\mathrm{Der}^g (R, \bar{\Q}) \to \mathrm{Der}^f (R_0, \bar{\Q})
$$
is bijective. Indeed, it is obviously injective, and since any $\delta \in \mathrm{Der}^f (R_0, \bar{\Q})$ vanishes on the elements of the ideal $(X^3 - Y^2) R_0$, it is also surjective.
Thus, $\dim_{\bar{\Q}} \mathrm{Der}^g (R, \bar{\Q}) = 2.$ On the other hand, if $L$ is the fraction field of $R$, then $\ell := \text{tr. deg.}_{\Q}L$ is 1.

\section{Applications to rigidity}\label{S:SR}

In this section, we will show how our results from \cite{IR} imply various forms of classical rigidity for the
elementary groups $E (\Phi, \I)$, where $\Phi$ is a reduced irreducible root system of rank $> 1$ and $\I$ is a ring of algebraic integers (or $S$-integers) in a number field.
It is worth mentioning that all forms of rigidity ultimately boil down to the fact that $\I$ does not admit nontrivial derivations. 

To fix notations, let $\Phi$ be a reduced irreducible root system of rank $>1$, $G$ the universal Chevalley-Demazure group scheme of type $\Phi$, and $\I$ a ring of algebraic $S$-integers in a number field $L$ such that $(\Phi, \I)$ is a nice pair. 
Furthermore, let $\Gamma = E(\Phi, \I)$ be the elementary subgroup of $G(\I).$

\begin{prop}\label{SR:P-1}
Let $\rho \colon \Gamma \to GL_m (K)$ be an abstract linear representation over an algebraically closed field $K$ of characteristic 0. Then there exist

\vskip1mm

\noindent {\rm (i)} \parbox[t]{16cm}{a finite dimensional commutative $K$-algebra
$$
A \simeq K^{(1)} \times \cdots \times K^{(r)},
$$
with $K^{(i)} \simeq K$ for all $i$;}

\vskip1mm

\noindent {\rm (ii)} \parbox[t]{15cm}{a ring homomorphism $f = (f^{(1)}, \dots, f^{(r)}) \colon \I \to A$ with Zariski-dense image, where each $f^{(i)} \colon \I \to K^{(i)}$ is the restriction to $\I$ of an embedding $\varphi_i \colon L \hookrightarrow K$, and $\varphi_1, \dots, \varphi_r$ are all \emph{distinct}; and}

\vskip1mm

\noindent {\rm (iii)} \parbox[t]{13cm}{a morphism of algebraic groups $\sigma \colon G(A) \to GL_m (K)$}

\vskip1mm

\noindent such that for a suitable subgroup
of finite index $\Delta \subset \Gamma$, we have
$$
\rho \vert_{\Delta} = \sigma \vert_{\Delta}.
$$
\end{prop}
\begin{proof}
Let $H = \overline{\rho (\Gamma)}$, where, as before, the bar denotes Zariski closure. We begin by showing that the connected component $H^{\circ}$ is automatically reductive. Suppose this is not the case and let $U$ be the unipotent radical of $H^{\circ}.$ Since the commutator subgroup $U' = [U,U]$ is a closed normal subgroup of $H$, the quotient $\check{H} = H/U'$ is affine, so we have a closed embedding $\iota \colon \check{H} \to GL_{m'} (K)$ for some $m'$.
Then, $\check{\rho} = \iota \circ \pi \circ \rho,$ where $\pi \colon H \to \check{H}$ is the quotient map, is a linear representation of $\Gamma$ such that $\overline{\check{\rho} (\Gamma)}^{\circ} = \check{H}^{\circ}$ has commutative unipotent radical. So, we can now apply (\cite{IR}, Theorem 6.7) to obtain a finite-dimensional commutative
$K$-algebra $\check{A}$, a ring homomorphism $\check{f} \colon \I \to \check{A}$ (which is injective as any nonzero ideal in $\I$ has finite index) with Zariski-dense image, and a morphism $\check{\sigma} \colon G(\check{A}) \to \check{H}$ of algebraic groups such that for a suitable finite-index subgroup $\check{\Delta} \subset \Gamma$, we have
$$
\check{\rho} \vert_{\check{\Delta}} = (\check{\sigma} \circ \check{F})\vert_{\check{\Delta}},
$$
where $\check{F} \colon \Gamma \to G(\check{A})$ is the group homomorphism induced by $\check{f}.$

Now let $J$ be the Jacobson radical of $\check{A}$. Since $\check{H}^{\circ}$ has commutative unipotent radical, $J^2 = \{ 0 \}$ by (\cite{IR}, Lemma 5.7). We claim that in fact $J = \{ 0 \}.$
Indeed, using the Wedderburn-Malcev Theorem as in the proof of Lemma \ref{D:L-2}, we can write $\check{A} = \oplus_{i=1}^r \check{A}_i$, where for each $i$, $\check{A}_i = K \oplus J_i$ is a finite-dimensional local $K$-algebra with maximal ideal $J_i$ such that $J_i^2 = \{ 0 \}.$ Then it suffices to show that $J_i = \{ 0 \}$ for all $i$.
So, we may assume that $\check{A}$ is itself a local $K$-algebra of this form. Then, fixing a $K$-basis $\{\varepsilon_1, \dots, \varepsilon_s \}$ of $J$, we have
$$
\check{f}(x) = f_0(x) + \delta_1 (x) \varepsilon_1 + \cdots + \delta_s (x) \varepsilon_s,
$$
where $f_0 \colon \I \to K$ is an injective ring homomorphism and $\delta_1, \dots, \delta_s \in \mathrm{Der}^{f_0} (\I, K).$
On the other hand, since the fraction field of $\I$ is a number field, it follows from Lemma \ref{D:L-Der} that the
derivations $\delta_1, \dots, \delta_s$ are identically zero. So, the fact that
$\check{f}$ has Zariski-dense image forces $J = \{ 0 \}.$ Consequently,
$\check{A} \simeq K \times \cdots \times K.$

Now by (\cite{IR}, Proposition 5.3), $\check{\sigma} \colon G(\check{A}) \to \check{H}^{\circ}$ is surjective,
so $\check{H}^{\circ}$ is semisimple, in particular reductive (\cite{Bo}, Proposition 14.10). It follows that $U = [U, U]$ (cf. \cite{Bo}, Corollary 14.11), and hence, being a nilpotent group, $U = \{ e \}$, which contradicts our original assumption. Thus, $H^{\circ}$ must be reductive, as claimed.

We can now apply (\cite{IR}, Theorem 6.7) to $\rho$ to obtain
a finite-dimensional commutative $K$-algebra $A$, a ring homomorphism $f \colon \I \to A$ with Zariski-dense image, and a morphism $\sigma \colon G(A) \to H$ of algebraic groups such that for a suitable subgroup of finite index $\Delta \subset \Gamma$, we have
$$
\rho \vert_{\Delta} = (\sigma \circ F) \vert_{\Delta}.
$$
Moreover, the fact that $H^{\circ}$ is reductive implies that
$A = K \times \cdots \times K$ (cf. \cite{IR}, Proposition 2.20 and Lemma 5.7). So, we can write
$f = (f^{(1)}, \dots, f^{(r)})$, for some ring homomorphisms $f^{(1)}, \dots, f^{(r)} \colon \I \to K$.
It is easy to see that all of the $f^{(i)}$ are injective,
and since $L$ is the fraction field of $\I$, it follows that each homomorphism $f^{(i)}$ is a restriction to $\I$ of an embedding  $\varphi_i \colon L \hookrightarrow K.$ Finally, since $f$ has Zariski-dense image,
all of the $\varphi_i$ must be {\it distinct}. This completes the proof.
\end{proof}

\vskip2mm

\noindent Keeping the notations of the proposition, we have the following

\begin{cor}\label{SR:C-1}
Any representation $\rho \colon \Gamma \to GL_m (K)$ is completely reducible.
\end{cor}
\begin{proof}
By Proposition \ref{SR:P-1}, we have $\rho \vert_{\Delta} = \sigma \vert_{\Delta},$ so since $G(B)$ is a semisimple group and char~$K = 0,$ $\rho \vert_{\Delta}$ is completely reducible. Since $\Delta$ is a finite-index subgroup of $\Gamma$, it follows that $\rho$ is also completely reducible.
\end{proof}

\vskip5mm

\noindent {\it $SS$-rigidity and local rigidity.} Notice that since by Lemma \ref{D:L-Der} there are no nonzero derivations $\delta \colon \I \to K$, Proposition \ref{D:P-1} and the estimate given in (\ref{E:D-TSpDim}) yield that for $\Gamma = E(\Phi, \I),$ we have $\dim X_n (\Gamma) = 0$ for all $n \geq 1$, i.e.
$\Gamma$ is $SS$-rigid. In fact, Corollary \ref{SR:C-1} implies that
$\Gamma$
is {\it locally rigid}, that is
$H^1 (\Gamma, \mathrm{Ad} \circ \rho) = 0$ for {\it any} representation $\rho \colon \Gamma \to GL_m (K)$.
This is shown in \cite{LM}, and we recall the argument for the reader's convenience. Let
$V = K^m.$ It is well-known that
$$
H^1 (\Gamma, \mathrm{End}_K (V,V)) = \mathrm{Ext}_{\Gamma}^1 (V,V)
$$
(cf. \cite{LM}, pg. 37), and
$\mathrm{Ext}_{\Gamma}^1 (V,V) = 0$ by Corollary \ref{SR:C-1}. But $\mathrm{Ad} \circ \rho$, whose underlying vector space is $M_m (K)$, can be naturally identified as a $\Gamma$-module with $\mathrm{End}_K (V,V),$ so $H^1 (\Gamma, \mathrm{Ad} \circ \rho) = 0$, as claimed.

\vskip4mm

\noindent {\it Superrigidity} (cf. \cite{BMS}, \S 16, and \cite{Ma}, Ch. VII). \ Let $\Gamma = SL_n (\Z)$ ($n \geq 3$) and consider an abstract representation $\rho \colon \Gamma \to GL_m (K).$ Then there exists a rational representation $\sigma \colon SL_n (K) \to GL_m (K)$ such that
$$
\rho \vert_{\Delta} = \sigma \vert_{\Delta}
$$
for a suitable finite-index subgroup $\Delta \subset \Gamma.$ Indeed, let $f \colon \Z \to A$ be the homomorphism associated to $\rho$. Since $A \simeq K^{(1)} \times \cdots \times K^{(r)}$ by Proposition \ref{SR:P-1}, we see that
$f$ is simply a diagonal embedding of $\Z$ into $K \times \cdots \times K.$ But $f$ has Zariski-dense image, so $r = 1,$ and the rest follows.

Notice that for a general ring of $S$-integers $\I$, the algebraic group $G(A)$ that arises in Proposition \ref{SR:P-1} can be described as follows. Let $\G = R_{L/\Q} (_{L}\!G)$, where $_{L}\!G$ is the algebraic group obtained from $G$ by extending scalars from $\Q$ to $L$ and $R_{L/\Q}$ is the functor of restriction of scalars. Then $\G(K) \simeq G(K) \times \cdots \times G(K)$, with the factors corresponding to all of the distinct embeddings of $L$ into $K$ (cf. \cite{PR}, \S 2.1.2). The group $G(A)$ is then obtained from $\G(K)$ by simply projecting to the factors corresponding to the embeddings $\varphi_1, \dots, \varphi_r$, so any representation of $E(\Phi, \I)$ factors through $\G.$

\vskip3mm

\noindent {\bf Remark 5.3.} Let us point out that another situation in which $\mathrm{Der}^{f} (R,K) = 0$ occurs is if $K$ is a field of characteristic $p >0$ and $R$ is a commutative ring of characteristic $p$ such that $R^p = R.$
This allows one to use arguments similar to the ones presented in this section to recover results of Seitz \cite{S}.
Details will be published elsewhere.

\vskip5mm

\bibliographystyle{amsplain}

\end{document}